\documentclass[12pt]{amsart}
\usepackage[babel]{csquotes}
\usepackage{enumitem}
\usepackage{amsmath,amsthm,amssymb,mathrsfs,amsfonts,verbatim,enumitem,color,leftidx}
\usepackage{kotex}
\usepackage{mathabx}
\usepackage{graphicx}
\usepackage[left=2.9cm,right=2.9cm,top=3.3cm,bottom=3.4cm,a4paper]{geometry}
\usepackage{etoolbox} 
\usepackage{tikz}
\usepackage{bbm}
\usepackage[all,tips]{xy}
\usepackage{graphicx,ifpdf}
\usepackage{stmaryrd}
\usepackage{multirow}
\usepackage{tikz-cd}

\ifpdf
\fi

\usepackage[right,displaymath,mathlines]{lineno}
\allowdisplaybreaks[0]

\definecolor{dartmouthgreen}{rgb}{0.05, 0.5, 0.06}

\usepackage[colorlinks]{hyperref}
\hypersetup{
linkcolor=blue,          
citecolor=dartmouthgreen,      
}

\usepackage{tikz}
\usepackage{here}
\usetikzlibrary{arrows,snakes,backgrounds}
\usetikzlibrary{matrix,arrows}

\newtheorem{thm}{Theorem}[section]
\newtheorem{lem}[thm]{Lemma}
\newtheorem{coro}[thm]{Corollary}
\newtheorem{prop}[thm]{Proposition}

\theoremstyle{definition}
\newtheorem{defn}[thm]{Definition}

\theoremstyle{remark}

\numberwithin{equation}{section}

\definecolor{esperance}{rgb}{0.0,0.5,0.0}

\title[Spectrum of non-uniform quotient of $PGL_4$]{Weak Ramanujan property of the standard non-uniform arithmetic quotient of $PGL_4$}

\begin{document}

\begin{abstract}
Let $F$ be a field of formal series over a finite field and $\mathcal{B}_d$ be the affine building associated to $PGL_d(F)$. Given a lattice $\Gamma$ in $PGL_d(F)$, the complex arising as a quotient $\Gamma\backslash \mathcal{B}_d$ is called weakly Ramanujan if every non-tivial discrete simultaneous spectrum of the colored adjacency operators $A_1,A_2,\ldots,A_{d-1}$ acting on $L^2(\Gamma\backslash \mathcal{B}_d)$ is contained in the simultaneous spectrum of those operators acting on $L^2(\mathcal{B}_d)$. In this paper, we prove that the standard non-uniform arithmetic quotient $PGL_4(\mathbb{F}_q[t])\backslash \mathcal{B}_4$ of $PGL_4(F)$ is weakly Ramanujan.
\end{abstract}


\author{Soonki Hong}
\address{Soonki Hong}
\curraddr{Department of Mathematics\\ Postech \\ Cheongam-ro, Pohang, 37673 \\ Republic of Korea}
\email{soonkihong@postech.ac.kr}

\author{Sanghoon Kwon}
\address{Sanghoon Kwon*}
\curraddr{Jinsil Building Room 506, Department of Mathematical Education\\ Catholic Kwandong University \\ Gangneung 25601 \\ Republic of Korea}
\email{shkwon1988@gmail.com \\ skwon@cku.ac.kr}



\maketitle
\tableofcontents


\section{Introduction}\label{sec:1}

A finite $k$-regular graph $X$ is called a \emph{Ramanujan graph} if for every eigenvalue $\lambda$ of the adjacency matrix $A_X$ of $X$ satisfies either $\lambda=\pm k$ or $|\lambda|\le 2\sqrt{k-1}$. An eigenvalue $\lambda$ is called trivial if $\lambda=\pm k$. Since the interval $[-2\sqrt{k-1},2\sqrt{k-1}]$ is equal to the spectrum of the adjacency operator of $k$-regular tree $\mathcal{T}_k$, we note that a finite $k$-regular graph $X$ is Ramanujan if and only if every {non-trivial} spectrum of $A_X$ is contained in the spectrum of the adjacency operator $A$ on $L^2(\mathcal{T}_{k})$.

Ramanujan graphs can be constructed as quotients of the Bruhat-Tits tree associated to $PGL(2,\mathbb{Q}_p)$ by congruence subgroups of uniform lattices of $PGL(2,\mathbb{Q}_p)$ (see \cite{LPS}), using Ramanujan conjecture for classical modular forms. More examples were given by Morgenstern \cite{Mo1}, replacing $\mathbb{Q}_p$ by $\mathbb{F}_q(\!(t^{-1})\!)$. One significant difference between $PGL(2,\mathbb{Q}_p)$ and $PGL_2(\mathbb{F}_q(\!(t^{-1})\!))$ is that $PGL(2,\mathbb{F}_q(\!(t^{-1})\!))$ has a non-uniform lattice $\Gamma=PGL(2,\mathbb{F}_q[t])$. For congruence subgroups $\Lambda$ of $\Gamma$, the quotient graphs are infinite but the edges and vertices come with suitable weights $w$ so that the total volume associated to the weight is finite. Under these weights on vertices and edges, the adjacency operator $A$ on the $(q+1)$-regular tree $\mathcal{T}_{q+1}$ induces the weighted adjacency operator $A_{X}$ on the quotient $X=\Lambda\backslash\mathcal{T}_{q+1}$. 

In \cite{M}, the author defined \emph{Ramanujan diagrams} as wighted objects that satisfies the similar bound for non-trivial spectrum of $A_X$. A \emph{diagram} is a pair $(X,w)$ where $X$ is a bipartite graph and $w\colon VX\cup EX\to \left\{\frac{1}{n}\colon n=1,2,3,\ldots\right\}$ is the \emph{weight function} satisfying
$$\sum_{v\in VX}w(v)<\infty\qquad\textrm{ and }\qquad  \frac{w(u)}{w(e)},\frac{w(v)}{w(e)}\in\mathbb{Z}\textrm{ for all }e=(u,v)\in EX.$$
Let $f,g$ be functions on $VX$ and set
$$\langle f,g\rangle =\sum_{v\in VX}f(v)\overline{g(v)}w(v)$$ which defines $L^2_w(X)$, the space of functions $f$ for which $\langle f,f\rangle<\infty$, as a Hilbert space. A diagram $(X,w)$ is called a {Ramanujan diagram} if every non-trivial spectrum of the weight adjacency operator $A_X$ on $L^2_w(X)$ is contained in the interval $[-2\sqrt{q},2\sqrt{q}]$. For example, the diagram $(PGL(2,\mathbb{F}_q[t])\backslash \mathcal{T}_{q+1},w)$ with the weights
$$w(v)=\frac{1}{|\Gamma_v|}\qquad \textrm{ and }\qquad w(e)=\frac{1}{|\Gamma_e|} $$ is a Ramanujan diagram since the adjacency operator on $L^2_w(PGL(2,\mathbb{F}_q[t])\backslash\mathcal{T}_{q+1})$ has discrete spectrum $\pm(q+1)$ and continuous spectrum $[-2\sqrt{q},2\sqrt{q}]$ (see Figure~\ref{Spec}).

\vspace{1em}
\begin{center}
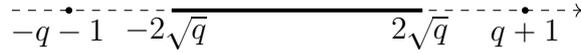
\begin{figure}[h]
\begin{tikzpicture}[scale=0.75]
        \draw[dashed,->] (-5,0) -- (5,0);

        \draw[line width=1.5pt,black] (2.2,0) -- (-2.2,0);
            
        \draw[fill] (4,0) circle (0.05);
        \draw[fill] (-4,0) circle (0.05);
        
        \node at (4,-0.4) {$q+1$};
        \node at (-4.2,-0.4) {$-q-1$};
        
        \node at (2.15,-0.4) {$2\sqrt{q}$};
        \node at (-2.3,-0.4) {$-2\sqrt{q}$};     
\end{tikzpicture}
\caption{Spectrum of $A_X$ on $L^2_w(PGL(2,\mathbb{F}_q[t])\backslash \mathcal{T}_{q+1})$}\label{Spec}
\end{figure}
\end{center}

\vspace{-1em}


The authors in \cite{CSZ} suggested a generalization of the notion of Ramanujan graphs to the simplicial complexes. They considered simplicial complexes obtained as finite quotients of the Bruhat-Tits building $\mathcal{B}_d$ associated to $PGL(d,F)$ for a non-Archimedean local field $F$. The \textit{colored adjacency operator} $A_j:L^2(\mathcal{B}_d)\rightarrow L^2(\mathcal{B}_d)$ is defined for $f\in L^2(\mathcal{B}_d)$ by
$$ A_j f(x)=\sum_{\substack{ y\sim x\\ \tau(y)=\tau(x)+j}}f(y),$$
 where $y\sim x$ implies that there is an edge between $y$ and $x$ in $\mathcal{B}_d$ and $\tau\colon \mathcal{B}_d^{(0)}\to\mathbb{Z}/d\mathbb{Z}$ is a color function (see Section~\ref{sec:2} for the precise definition). Let $\mathcal{S}^d$ be the simultaneous spectrum of the colored adjacency operators $(A_1,\ldots,A_{d-1})$ on $L^2(\mathcal{B}_d)$, which may be computed explicitly as a subset of $\mathbb{C}^{d-1}$ (Theorem~2.11 of \cite{LSV1}, see also \cite{Ma}). In fact, if the cardinality of the residue field of $F$ is $q$, then $\mathcal{S}^d$ is equal to the set $\sigma(\mathcal{S})$ for 
 $$\mathcal{S}=\{(z_1,\ldots,z_d)\colon |z_1|=\cdots=|z_d|=1\textrm{ and }z_1z_2\cdots z_d=1\}$$ and the map $\sigma\colon S\to\mathbb{C}^{d-1}$ given by $(z_1,\ldots,z_d)\mapsto (\lambda_1,\ldots,\lambda_{d-1})$ where
 \begin{equation}\label{lambda-para}\lambda_k=q^{\frac{k(d-k)}{2}}\sigma_k(z_1,z_2,\ldots,z_d).\end{equation}
A finite complex $X$ arising as a quotient of $\mathcal{B}_d$ is called \textit{Ramanujan} if every non-trivial automorphic spectrum $(\lambda_1,\ldots,\lambda_{d-1})$ of $A_{X,j}$ acting on $L^2(X)$ is contained in the simultaneous spectrum $\mathcal{S}^d$ of $A_j$ on $L^2(\mathcal{B}_d)$.
 In \cite{LSV1}, \cite{LSV2}, \cite{Winnie} and \cite{Sar}, the authors constructed higher dimensional Ramanujan complexes arising as finite quotients of $PGL(d,F)$. 

In \cite{S}, the author investigated non-uniform Ramanujan quotients of the Bruhat-Tits building $\mathcal{B}_d$ associated to $PGL(d,\mathbb{F}_q(\!(t^{-1})\!))$, generalizing the concept of finite Ramanujan complexes. She proved using the representation-theoretic argument that if $d>2$, then for $G=PGL(d,\mathbb{F}_q(\!(t^{-1})\!))$, $\Gamma=PGL(d,\mathbb{F}_q[t])$ and the Bruhat-Tits building $\mathcal{B}_d$ of $G$, the quotient $\Lambda\backslash \mathcal{B}_d$ is not Ramanujan for any finite index subgroup $\Lambda$ of $\Gamma$. In fact, the Ramanujan conjecture in positive characteristic for $PGL_d$ for $d>2$, achieved by Lafforgue, gives bounds on the cuspidal spectrum, but the other parts of the spectrum do not satisfy the same bounds as the cuspidal spectrum. See Figure~\ref{Spectrum} for the spectrum of $A_1$ on $PGL(3,\mathbb{F}_q[t])\backslash\mathcal{B}_3$, obtained in \cite{HK}.

\begin{center}
\begin{figure}[h]
\begin{tikzpicture}[scale=0.75]
        \draw[dashed,->] (-5,0) -- (5,0);
        \draw[dashed,->] (0,-5) -- (0,5);
        \def\a{0.5} \def\b{1.5}
        
        \draw[fill,line width=1pt,black] plot[samples=100,domain=0:360,smooth,variable=\t] ({(\b-\a)*cos(\t)+\a*cos((\b-\a)*\t/\a},{(\b-\a)*sin(\t)-\a*sin((\b-\a)*\t/\a});
        
        \def\a{1.2} \def\b{3.6}
        \draw[line width=1pt,black] plot[samples=100,domain=0:360,smooth,variable=\t] ({(\b-\a)*cos(\t)+\a*cos((\b-\a)*\t/\a},{(\b-\a)*sin(\t)-\a*sin((\b-\a)*\t/\a});
        
        \draw[fill] (4.5,0) circle (0.05);
        \draw[fill] (120:4.5) circle (0.05);
        \draw[fill] (240:4.5) circle (0.05);
        
        \node at (1.5,1.5) {$3q$};
        \draw[->] (1.5,1.3) -- (1.5,0.1);
        \node at (3.6,1.2) {$\sqrt{q}(q+\sqrt{q}+1)$};
        \draw[->] (3.6,1) -- (3.6,0.1);
        \node at (4.5,-1) {$q^2+q+1$};
        \draw[->] (4.5,-0.8) -- (4.5,-0.1);
        
\end{tikzpicture}
\caption{Spectrum of $A_1$ on $L^2_w(PGL(3,\mathbb{F}_q[t])\backslash \mathcal{B}_{3})$}\label{Spectrum}
\end{figure}
\end{center}

\vspace{-1em}

From a representation-theoretic perspective, one may consider the property of the complex $\Lambda\backslash\mathcal{B}_d$ being {weakly Ramanujan}. Given a lattice $\Gamma$ in $PGL_d(\mathbb{F}_q(\!(t^{-1})\!))$, the simplicial complex arising as the quotient $\Gamma\backslash \mathcal{B}_d$ is called \emph{weakly Ramanujan} if every non-trivial \emph{discrete} simultaneous spectrum of the colored adjacency operators $A_{w,1},A_{w,2},\ldots,A_{w,d-1}$ with natural weights acting on $L^2_w(\Gamma\backslash  \mathcal{B}_d)$ is contained in the simultaneous spectrum $\mathcal{S}^d$ of $A_1,A_2,\ldots,A_{d-1}$ on $L^2(\mathcal{B}_d)$. For every prime $d$ with $d\ge 3$, it is known that $\Gamma\backslash \mathcal{B}_d$ is weakly Ramanujan for every congruence subgroup $\Gamma$ of $PGL_d(\mathbb{F}_q[t])$. On the other hand, when $d$ is not a prime, there are infinitely many non-uniform quotients $\Gamma\backslash \mathcal{B}_d$ which are not weakly Ramanujan.

In this paper, we explore the simplest non-prime case. We provide a combinatorial characterization of the automorphic spectrum of the natural weighted adjacency operator on the non-uniform simplicial complex $PGL(4,\mathbb{F}_q[t])\backslash\mathcal{B}_4$ and prove that it is indeed weakly Ramanujan.

\begin{thm}\label{thm:wR}
The standard non-uniform arithmetic quotient $PGL_4(\mathbb{F}_q[t])\backslash \mathcal{B}_4$ of the building for $PGL(4,\mathbb{F}_q(\!(t^{-1})\!))$ is weakly Ramanujan.
\end{thm}

The following theorem describes the automorphic simultaneous spectrum of $A_{w,i}$, which yields Theorem~\ref{thm:wR}. We recall that the parametrization of $(\lambda_1,\lambda_2,\lambda_3)$ is given in equation~\eqref{lambda-para} by a complex 4-tuple $(z_1,z_2,z_3,z_4)$. 
\begin{thm}\label{thm:main}
If $(\lambda_1,\lambda_2,\lambda_3)$ is in the simultaneous spectrum of $A_i$, then $(z_1,z_2,z_3,z_4)$, up to permutation, belongs to one of the following cases:
\begin{itemize}
\item $(q\sqrt{q}e^{\frac{k\pi}{2}i},\sqrt{q}e^{\frac{k\pi}{2}i},\frac{1}{\sqrt{q}}e^{\frac{k\pi}{2}i},\frac{1}{q\sqrt{q}}e^{\frac{k\pi}{2}i})$, $k=0,1,2,3$;$\,\, \leftarrow\,\,$trivial spectrum
\item $(qe^{i\theta},e^{i\theta},e^{-3i\theta},\frac{1}{q}e^{i\theta})$, $\theta\in\mathbb{R}$;
\item $(\sqrt{q}e^{i\theta_1},\pm \sqrt{q}e^{-i\theta_1},\frac{1}{\sqrt{q}}e^{i\theta_1},\pm\frac{1}{\sqrt{q}}e^{-i\theta_1})$, $\theta_1\in\mathbb{R}$;
\item $(\sqrt{q}e^{i\theta_1},e^{i\theta_2},e^{-i(\theta_2+2\theta_1)},\frac{1}{\sqrt{q}}e^{i\theta_1})$, $\theta_1,\theta_2\in\mathbb{R}$;
\item $(e^{i\theta_1},e^{i\theta_2},e^{i\theta_3},e^{-i(\theta_1+\theta_2+\theta_3)})$, $\theta_1,\theta_2,\theta_3\in\mathbb{R}$.
\end{itemize}
\end{thm}
Theorem \ref{thm:main} includes the case that $z_i=z_j$ for some $i,j$ except the fourth case. 

The parametrization of $(\lambda_1,\lambda_2,\lambda_3)$ enable us to compute the simultaneous eigenfunction $f$ of $(\lambda_1,\lambda_2,\lambda_3)$ (see Section \ref{sec:5} and Appendix \ref{sec:appendix}). 
 For a given eigenfunction $f$ of a simultaneous eigenvalue $(\lambda_1,\lambda_2,\lambda_3)$, the function $(\ell,m,n)\mapsto f(v_{\ell,m,n})$ is expressed as the linear combination of $z_i^\ell z_j^m z_k^n$, where $v_{\ell,m,n}$ is an vertex in $\Gamma\backslash \mathcal{B}_d$ indexed by 3-tuple $(\ell,m,n)$ with $\ell\geq m\geq n\geq 0$ (see Section \ref{sec:3}). The coefficient of $z_i^\ell z_j^m z_k^n$ can be a polynomial with variables $\ell,m$ and $n.$

Let $f=f_{\textbf{z}}$ be the simultaneous eigenfunction of a simultaneous eigenvalue $(\lambda_1,\lambda_2,\lambda_3)$ parametrized by $\textbf{z}=(z_1,z_2,z_3,z_4)$ through the equation \eqref{lambda-para}. The proof of Theorem \ref{thm:main} comes from demonstrating the equivalence of following conditions:
\begin{itemize}
\item[$(A)$]\label{(a)} The 3-tuple $(\lambda_1,\lambda_2,\lambda_3)$ is contained in the simultaneous spectrum of $A_{w,i}$.
\item[$(B)$] For any $c>0$, we have 
$$f(v_{\ell,0,0})=(q^{(3+c)\ell/2}),f(v_{\ell,\ell,0})=O(q^{(2+c)\ell}),f(v_{\ell,\ell,\ell})=O(q^{({3}+c)\ell/2}).$$
\item[$(C)$] The 4-tuple $(z_1,z_2,z_3,z_4)$ belongs to the one of the cases in Theorem \ref{thm:main}.
\item[$(D)$] The coefficient of $z_i^\ell z_j^m z_k^n$ in $f(v_{\ell,m,n})$ is zero when $|z_i^\ell z_j^n z_k^n|>1$.
\end{itemize}
The following diagram summarizes the proof procedure that will appear in Section \ref{sec:6}:
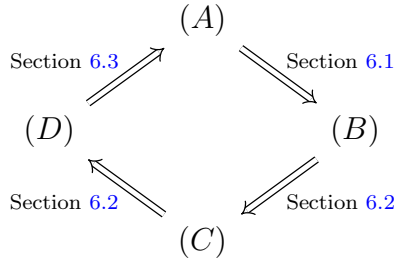
\[
\begin{tikzcd}
& (A)\arrow[Rightarrow]{dr}{\text{Section } \ref{sec:6.1}}   \\
(D)\arrow[Rightarrow]{ur}{\text{Section } \ref{sec:6.3}}& & (B)\arrow[Rightarrow]{dl}{\text{Section } \ref{sec:6.2}}\\
&(C)\arrow[Rightarrow]{ul}{\text{Section } \ref{sec:6.2}}&
\end{tikzcd}
 \] 
 
In fact, there is no non-trivial discrete automorphic simultaneous spectrum of $A_{w,i}$ on $L^2_w(PGL_4(\mathbb{F}_q[t])\backslash \mathcal{B}_4)$. We remark that there is a one-to-one correspondence between the set of 4-tuples $\{(|z_1|,|z_2|,|z_3|,|z_4|)\}$ and the set of integer partitions of 4 by matching $(q^{\frac{n-1}{2}},q^{\frac{n-3}{2}},\ldots,q^{\frac{1-n}{2}})$ to $n$ as follows.
\begin{center}
\begin{tabular}{ccc}
$\left(q\sqrt{q},\sqrt{q},\frac{1}{\sqrt{q}},\frac{1}{q\sqrt{q}}\right)$ &$\longleftrightarrow$ &4 \\
$\left(q,1,\frac{1}{q},q\right)$ &$\longleftrightarrow$ &3+1\\
$\left(\sqrt{q},\frac{1}{\sqrt{q}},\sqrt{q},\frac{1}{\sqrt{q}}\right)$ &$\longleftrightarrow$ & 2+2\\
$\left(\sqrt{q},\frac{1}{\sqrt{q}},1,1\right)$ &$\longleftrightarrow$ & 2+1+1\\
$(1,1,1,1)$ &$\longleftrightarrow$ & 1+1+1+1
\end{tabular}
\end{center}
 
 \subsection{A brief summary to the paper} Section~\ref{sec:2} reviews some notational conventions related to the Bruhat-Tits building for $PGL(d,F)$ and explains the weighted adjacency operators and their spectrum. We provide an explicit description of the $PGL(4,\mathbb{F}_q[t])$-quotient of the building $\mathcal{B}_4$ in Section~\ref{sec:3}. Section~\ref{sec:4} summarizes the calculations of natural weight functions for adjacency operators, based on the cardinality of stabilizers. Theorem~\ref{thm:eigenftn} describes simultaneous eigenfunctions for generic triples $(\lambda_1,\lambda_2,\lambda_3)$. In Section~\ref{sec:6}, we investigate which $(\lambda_1,\lambda_2,\lambda_3)$ are indeed automorphic spectrum. Appendix~\ref{sec:appendix} contains the proof of Theorem~\ref{thm:eigenftn} and the discussion of non-generic cases. In Appendix \ref{ap:b}, we prove Lemma \ref{lem:6.77}.

\subsubsection*{Acknowledgement} This work is supported by SSTF (Samsung Science and Technology Foundation) Project Number SSTF-BA2101-01 and NRF grant (No. RS-2023-00237811).

\section{Preliminaries: Buildings and colored adjacency operators}\label{sec:2}
Let $F$ be a non-Archmedean local field with a discrete valuation $\nu$ and $\mathcal{O}_F$ be the valuation ring of $F$ with the uniformizer $\pi$. Let $G$ be the projective general linear group $PGL(d,F)$ over $F$ given by
$$PGL(d,F)=GL(d,F)/\{\lambda I :\lambda\in F^\times\}$$
and denote by $W$ the image of the map $g\rightarrow g\{\lambda I:\lambda \in F^\times\}$ from $GL(d,\mathcal{O}_F)$ to $PGL(d,F)$. Two $\mathcal{O}_F$-modules $L$ and $L'$ of rank $d$ are said to be equivalent if $L=sL'$ for some $s\in F^\times$. Each coset $gW$ in $G/W$ can be interpreted as an equivalence class of $\mathcal{O}_F$-lattices, i.e. $\mathcal{O}_F$-modules of rank $d$.

The Bruhat-Tits building $\mathcal{B}_d$ associated to $PGL(d,F)$ is the $(d-1)$-dimensional contractible simplicial complex constructed as follows. The set of vertices $\mathcal{B}_d^{(0)}$ is the set of equivalence classes $[L]$ of $\mathcal{O}_F$-lattices. More generally, the set $\mathcal{B}_d^{(k)}$ of $k$-simplices is the set of $k+1$-vertices $[L_0]$, $[L_1],\cdots,[L_k]$ satisfying 
\begin{equation}\label{eq:2.1}
\pi L_1'\subset L_k'\subset L_{k-1}'\subset \cdots \subset L_2'\subset L_1'
\end{equation}
for some $L_i'\in[L_i].$ A maximal set of mutually adjacent vertices has cardinality $d$ and they form a $(d-1)$-dimensional simplicial complex, called \emph{chamber}. The chambers sharing a fixed vertex correspond bijectively to the complete flags in $(\mathcal{O}/\pi\mathcal{O})^d$.
 
The group $PGL(d,F)$ acts on $\mathcal{B}_d$ by left translation, preserving $k$-dimensional facets for all $0\le k\le d-1$. 
Since the group $W$ is the stabilizer of the vertex $[\mathcal{O}_F^d]$, the set $\mathcal{B}_d^{(0)}$ of vertices identified with the quotient space $G/W$ and we may denote by $[L_g]$ the vertex of $\mathcal{B}_d$ corresponding to the coset $gW$. Each vertex $[L_g]$ has a \emph{color} given by $\textrm{ord}_\pi\det(g)$ mod $d$. Equivalently, the color $\tau$ from $\mathcal{B}_d^{(0)}$ to $\mathbb{Z}/d\mathbb{Z}$ is defined by 
  \begin{equation}\label{eq:2.2}\tau([L]):=\log_q[\mathcal{O}_F^d:\pi^i L],\end{equation}
 for sufficiently large positive integer $i$ satisfying $\pi^i L\subset \mathcal{O}_F^d.$ Since $[\pi^i L:\pi^{i+1} L]=d$, the color of a vertex $[L]$ does not depend on the choice of a lattice in $[L]$. Adjacent vertices have different colors, and in particular, the vertices in a chamber exhaust all possible colors.

  Let $L^2(\mathcal{B}_d)$ be the space of functions $f$ on $\mathcal{B}_d$ satisfying
 $$\|f\|_2:=\biggl(\sum_{x\in \mathcal{B}_d^{(0)}}|f(x)|^2\biggr)^{1/2}<\infty.$$ 
  The \textit{colored adjacency operator} $A_i$ on $L^2(\mathcal{B}_d)$ is defined for any $L^2$-function $f$ by
  \begin{equation}\label{eq:2.3}
  A_if(x):=\sum_{\substack{y\sim x\\\tau(y)=\tau(x)+i}}f(y),
  \end{equation}
   where $y\sim x$ implies that $y$ is a neighborhood $x$. These operators are bounded, commutative, and the operators $A_i$ and $A_{d-i}$ are the adjoint of each other. Consequently, each $A_i$ is a normal operator.

Generalizing the idea of \cite{M}, the colored adjacency operators with weights are introduced in \cite{S}. Let $\Gamma$ be a  discrete subgroup of $PGL(d,F)$. Although the color function defined on $\mathcal{B}_d^{(0)}$ may not be preserved by $\Gamma$, the colors defined on the set $\mathcal{B}_d^{(1)}$ of edges (1-skeletons of $\mathcal{B}_d$) by 
$$\tau(x,y)=\tau(x)-\tau(y)\quad(\textrm{mod }d)$$ are preserved by $\Gamma$. Thus, the operators $A_i$ defined by \eqref{eq:2.3} induce $A_{X,i}$ on the space of functions on $\Gamma\backslash\mathcal{B}_d$. If $\Gamma$ acts on $\mathcal{B}_d$ with torsions, the induced operators on the quotient should come with weights envolving the cardinality of the stabilizer of the vertices and edges.

Given a simplicial complex $X$ (not necessarily finite), let $V=VX$ and $E=EX$ be the set of vertices and edges of $X$, respectively. Denote by $(u,v)$ the edge connecting $u$ and $v$. 
 \begin{defn} A function $w:V\cup E\rightarrow (0,1]$ is called a \textit{weight function} if $\frac{w(e)}{w(v)}$ at least 1 for any $v\in V$ and $e\in E$ with $v\in e$. The function $\theta (u,v):=\frac{w(e)}{w(u)}$ is called the \textit{entering degree} of $e=(u,v)$ to $u$. The \textit{in-degree} of a vertex $u$ is defined by $\text{in-degree}(u):=\sum_{(u,v)\in E}\theta(u,v)$. A simplicial complex $X$ with a weight function $w$ is $k$-regular if $\text{in-degree}(u)=k$ for any $u\in V$.
 \end{defn}
 Fix a weight function $w$. A measure $\mu$ on $V$ is defined by $\mu(S)=\underset{v\in S}\sum w(v)$ for any $S\subset V$. The measure $\mu$ allows to define $L^2$-norm of a function $f$ by
 $$\|f\|_{2,w}:=\biggl(\sum_{u\in V}|f(u)|^2w(u)\biggr)^{1/2}.$$
 Let $L^2_w(X)$ be the space of functions $f$ on $V$ satisfying $\|f\|_{2,w}<\infty.$ 
 
 \begin{defn}
 The \textit{weighted adjacency operators} $A_{w,i}$ on $L^2_w(\Gamma\backslash\mathcal{B}_d)$ is defined for any $f\in L^2_w(\Gamma\backslash\mathcal{B}_d)$ by  
 $$A_{w,i}f(u):=\sum_{\substack{(u,v)\in E\\\tau(v)=\tau(u)+i}} \frac{w(u,v)}{w(u)}f(v).$$
 \end{defn}
 With respect to the weight $w$ given by $$w(v)=\frac{1}{|\Gamma_v|}\qquad\textrm{ and }\qquad w(e)=\frac{1}{|\Gamma_e|},$$ the colored adjacency operators $A_i$ on $\mathcal{B}_d$ induce the weighted adjacency operators $A_{w,i}$ on $L^2_w(\Gamma\backslash \mathcal{B}_d)$.

The \textit{simultaneous spectrum} $\mathcal{S}^d$ of $A_1,\ldots,A_{d-1}$ is the subset of $d-1$-tuples $(\lambda_1,\cdots,\lambda_{d-1})$ in $\mathbb{C}^{d-1}$ for which there exists a sequence $\{f_n\}$ of $L^2$-functions with $\|f_n\|_2=1$ such that for any $i=1,\cdots,d-1$,
   $$\lim_{n\rightarrow \infty} \|A_if_n-\lambda_if_n\|_2=0.$$
  According to MacDonald \cite{Ma} (or see also Theorem~2.11 of \cite{LSV1}), the spectrum of $A_i$ for $1\le i\le d-1$ on $\mathcal{B}_d$ is $q^{i(d-i)}\Omega_{d,i}$ where
 $$\Omega_{d,i}=\{\sigma_i(z_1,\ldots,z_d)\colon |z_1|=\cdots=|z_d|=1\textrm{ and }z_1z_2\cdots z_d=1\}$$ and 
 $\sigma_i(z_1,\ldots,z_d)=\sum_{1\le k_1<\cdots<k_i\le d} z_{k_1}z_{k_2}\cdots z_{k_i}$ is the $i$-th elementary symmentric polynomial.


\section{$PGL(4,\mathbb{F}_q[t])$ quotient of the building $\mathcal{B}_4$}\label{sec:3}

Let $\mathbb{F}_q$ be a finite field of order $q$. Let $\mathbb{F}_q[t]$ and $\mathbb{F}_q(t)$ be the ring of polynomials and the field of rational functions over $\mathbb{F}_q$, respectively. Consider the absolute value $\|\cdot\|$ of $\mathbb{F}_q(t)$ given for any $f,g\in \mathbb{F}_q[t]$ by $$\left\|\frac{f}{g}\right\|:=q^{\deg(f)-\deg(g)}.$$
The completion $F$ of $\mathbb{F}_q(t)$ with respect to $\|\cdot\|$ is the field of formal Laurent series in $t^{-1}$
$$\mathbb{F}_q(\!(t^{-1})\!):=\left\{\sum_{n=-N}^\infty a_nt^{-n}:N\in \mathbb{Z},a_n\in \mathbb{F}_q\right\}$$
and the valuation ring $\mathcal{O}_F$ of $\mathbb{F}_q(\!(t^{-1})\!)$ is the subring of power series 
$$\mathbb{F}_q[\![t^{-1}]\!]:=\left\{\sum_{n=0}^{\infty} a_nt^{-n}:n\in \mathbb{Z},a_n\in\mathbb{F}_q\right\}.$$
Let $G:=PGL(4,\mathbb{F}_q(\!(t^{-1})\!))$, $W:=PGL(4,\mathbb{F}_q[\![t^{-1}]\!])$ and $\Gamma:=PGL(4,\mathbb{F}_q[t])
$. In this section, we give an explicit description of the standard fundamental domain for $\Gamma\backslash G/W$ and the weighted adjacency operators $A_{w,i}$ on $L^2_w(\Gamma\backslash \mathcal{B}_4)$ induced from the operators $A_i$ on $L^2(\mathcal{B}_4)$.

Let $x_{\ell,m,n}$ be the diagonal matrix $\textrm{diag}(t^\ell,t^m,t^n,1)$.
 Using Lemma~3.1 and the proof of Lemma~3.2 in \cite{HK}, we have the following lemma.
 \begin{lem}\label{lem:3.1} For any $g\in G$, there exists a unique $3$-tuple $(\ell,m,n)$ of integers with $\ell\ge m\ge n\ge 0$ such that 
 \begin{equation}\nonumber
 g\in \Gamma x_{\ell,m,n} W.
 \end{equation}
 \end{lem}
\begin{proof}
We note that every $g\in G$ may be written by 
$$\begin{pmatrix}
t^{\ell_0} & a_{12} & a_{13} & a_{14} \\ 0 & t^{m_0} & a_{23} & a_{24} \\ 0 & 0 & t^{n_0} & a_{34} \\ 0 & 0 & 0 & t^{k_0}
\end{pmatrix}w$$ for some $a_{ij}\in\mathbb{F}_q(\!(t^{-1})\!)$, matrix $w\in W$ and non-negative integers $\ell_0,m_0,n_0,k_0$. By multiplying suitable $\gamma$ on the left if necessary, we may assume that $\ell_0\ge m_0\ge n_0\ge k_0$.

By similar argument of the proof of Lemma~3.2 in \cite{HK}, we may find either a pair $(\gamma_1,w_1)\in \Gamma\times W$ such that 
$$\begin{pmatrix}
t^{\ell_0} & a_{12} & a_{13} & a_{14} \\ 0 & t^{m_0} & a_{23} & a_{24} \\ 0 & 0 & t^{n_0} & a_{34} \\ 0 & 0 & 0 & t^{k_0}
\end{pmatrix}=\gamma_1
\begin{pmatrix}
t^{\ell_1} & 0 & \star & \star \\ 0 & t^{m_1} & 0 & 0 \\ 0 & 0 & t^{n_1} & 0 \\ 0 & 0 & 0 & t^{k_1}
\end{pmatrix}w_1$$ or a pair $(\gamma_2,w_2)\in \Gamma\times W$ such that
$$\begin{pmatrix}
t^{\ell_0} & a_{12} & a_{13} & a_{14} \\ 0 & t^{m_0} & a_{23} & a_{24} \\ 0 & 0 & t^{n_0} & a_{34} \\ 0 & 0 & 0 & t^{k_0}
\end{pmatrix}=\gamma_2
\begin{pmatrix}
t^{\ell_2} & 0 & 0 & \star \\ 0 & t^{m_2} & 0 & \star \\ 0 & 0 & t^{n_2} & 0 \\ 0 & 0 & 0 & t^{k_2}
\end{pmatrix}w_2$$
for some $\ell_i\ge m_i\ge n_i\ge k_i$. In both cases, by applying Lemma~3.2 in \emph{loc. cit.} to the suitable $3\times 3$ block, we may find $\gamma$ and $w$ for which each case reduces to  
$$\gamma\begin{pmatrix}
t^\ell & 0 & 0 & 0 \\ 0 & t^{m} & 0 & 0 \\ 0 & 0 & t^{n} & 0 \\ 0 & 0 & 0 & t^{k}
\end{pmatrix}w.$$
This completes the proof.
\end{proof}
 Let $v_{\ell,m,n}$ be the vertex of the quotient complex $\Gamma\backslash \mathcal{B}_d$ corresponds to $\Gamma x_{\ell,m,n} W.$ By definition of the color (see equation~\ref{eq:2.2}), we have $$\tau(v_{\ell,m,n})=\ell+m+n\quad(\text{mod} \,4)$$ for every $v_{\ell,m,n}\in \mathcal{B}_d^{(0)}$.  
  
Let us denote by $V_{\ell,m,n}$ the set of all vertices adjacent to $v_{\ell,m,n}$ and let 
$$V_{\ell,m,n}^i=\{v\in V_{\ell,m,n}\colon \tau(v)=\tau(v_{\ell,m,n})+i\}.$$ Using \eqref{eq:2.1} and Lemma~\ref{lem:3.1}, we have the following Figure~\ref{figure} and Table~\ref{t1}, which present a part of a fundamental domain for $\Gamma\backslash \mathcal{B}_4$ and describes the neighborhoods of vertices, respectively.
\begin{figure}[h]
\begin{center}
\begin{tikzpicture}[scale=0.7]
\fill (4,-{sqrt(75)}+1.5) circle (3pt); 
\fill[red] (4,-{sqrt(75)}+1) circle (3pt); 
\fill[green] (4,-{sqrt(75)}+0.5) circle (3pt); 
\fill[blue] (4,-{sqrt(75)}) circle (3pt); 
\fill (-9,+{sqrt(3)}/3-1.5) circle (3pt);\fill[red] (-10,-1.5) circle (3pt); \fill[green] (-8,-1.5) circle (3pt); \fill[blue] (-9,+{sqrt(3)}-1.5)circle (3pt); 
\fill[green] (-6,-1.5) circle (3pt); \fill[blue] (-4,-1.5) circle (3pt); \fill (-2,-1.5) circle (3pt);  \fill (-5,+{sqrt(3)}-1.5) circle (3pt); \fill[red] (-3,+{sqrt(3)}-1.5) circle (3pt);\fill[green] (-4,+{sqrt(12)}-1.5) circle (3pt); \fill[blue] (-4,+{sqrt(48)}/3-1.5) circle (3pt); \fill [red] (-5,+{sqrt(3)}/3-1.5) circle (3pt); \fill[green] (-3,+{sqrt(3)}/3-1.5) circle (3pt);
\fill [blue] (0,-1.5) circle (3pt);\fill (2,-1.5) circle (3pt); \fill[red] (4,-1.5) circle (3pt);\fill [green] (6,-1.5) circle (3pt); \fill [blue] (5,+{sqrt(3)}-1.5) circle (3pt); \fill [green] (3,+{sqrt(3)}-1.5) circle (3pt); \fill [red] (1,+{sqrt(3)}-1.5) circle (3pt); \fill [blue] (2,+{sqrt(12)}-1.5) circle (3pt); \fill [red] (3,+{sqrt(27)}-1.5) circle (3pt); \fill [green] (1,+{sqrt(3)}/3-1.5) circle (3pt); \fill [blue] (3,+{sqrt(3)}/3-1.5) circle (3pt); \fill (5,+{sqrt(3)}/3-1.5) circle (3pt); \fill [red] (4,+{sqrt(48)}/3-1.5) circle (3pt); \fill [green] (3,+{sqrt(147)}/3-1.5) circle (3pt); \fill (2,+{sqrt(48)}/3-1.5) circle (3pt); \fill [green] (1,+{sqrt(3)}/3-1.5) circle (3pt); 
\fill (4,+{sqrt(12)}-1.5) circle (3pt);
\fill (-2,-{sqrt(12)}) circle (3pt); \fill [blue] (-1,-{sqrt(27)}) circle (3pt); \fill[green] (0,-{sqrt(48)}) circle (3pt); \fill [red] (1,-{sqrt(75)}) circle (3pt); \fill (2,-{sqrt(108)}) circle (3pt); \fill [blue] (0,-{sqrt(108)}) circle (3pt); \fill [green] (-2,-{sqrt(108)}) circle (3pt); \fill [red] (-4,-{sqrt(108)}) circle (3pt); \fill (-6,-{sqrt(108)}) circle (3pt); \fill [green] (-5,-{sqrt(75)}) circle (3pt); \fill (-4,-{sqrt(48)}) circle (3pt); \fill [green] (-3,-{sqrt(27)}) circle (3pt); \fill [red] (-2,-{sqrt(192)}/3) circle (3pt); \fill [blue] (-3,-{sqrt(363)}/3) circle (3pt); \fill [red] (-4,-{sqrt(588)}/3) circle (3pt); \fill [blue] (-5,-{sqrt(867)}/3) circle (3pt); \fill  (-3,-{sqrt(867)}/3) circle (3pt); \fill [red] (-1,-{sqrt(867)}/3) circle (3pt); \fill [green] (1,-{sqrt(867)}/3) circle (3pt); \fill [blue] (0,-{sqrt(588)}/3) circle (3pt);  \fill [green] (-2,-{sqrt(588)}/3) circle (3pt);\fill [blue] (-3,-{sqrt(75)}) circle (3pt); \fill (-1,-{sqrt(75)}) circle (3pt);  \fill [red] (-2,-{sqrt(48)}) circle (3pt); \fill (-1,-{sqrt(363)}/3) circle (3pt);
\draw [thick](-10,-1.5)--(-8,-1.5)--(-9,+{sqrt(3)}-1.5)--(-10,-1.5)--(-9,+{sqrt(3)}/3-1.5)--(-8,-1.5); \draw[thick] (-9,+{sqrt(3)}/3-1.5)--(-9,+{sqrt(3)}-1.5); 
\draw[thick] (-6,-1.5)--(-4,+{sqrt(12)}-1.5);\draw[thick] (-4,+{sqrt(12)}-1.5)--(-2,-1.5)--(-6,-1.5); \draw [dashed] (-5,+{sqrt(3)}-1.5)--(-4,-1.5)--(-3,+{sqrt(3)}-1.5)--(-5,+{sqrt(3)}-1.5); \draw [thick] (-5,+{sqrt(3)}/3-1.5)--(-3,+{sqrt(3)}/3-1.5)--(-4,+{sqrt(48)}/3-1.5)--(-5,+{sqrt(3)}/3-1.5); \draw [thick] (-6,-1.5)--(-5,+{sqrt(3)}/3-1.5)--(-5,+{sqrt(3)}-1.5)--(-4,+{sqrt(48)}/3-1.5)--(-4,+{sqrt(12)}-1.5)--(-4,+{sqrt(48)}/3-1.5)--(-3,+{sqrt(3)}-1.5)--(-3,+{sqrt(3)}/3-1.5)--(-2,-1.5)--(-3,+{sqrt(3)}/3-1.5)--(-4,-1.5)--(-5,+{sqrt(3)}/3-1.5); \draw [dashed] (-5,+{sqrt(3)}-1.5)--(-3,+{sqrt(3)}/3-1.5);
\draw[thick] (0,-1.5)--(6,-1.5)--(3,+{sqrt(27)}-1.5);\draw[thick](3,+{sqrt(27)}-1.5)--(0,-1.5); \draw[dashed] (5,+{sqrt(3)}-1.5)--(4,-1.5)--(3,+{sqrt(3)}-1.5)--(4,-1.5)--(1,+{sqrt(3)}-1.5)--(5,+{sqrt(3)}-1.5);\draw[dashed] (2,+{sqrt(12)}-1.5)--(3,+{sqrt(3)}-1.5)--(4,+{sqrt(12)}-1.5)--(2,+{sqrt(12)}-1.5);\draw[thick] (1,+{sqrt(3)}/3-1.5)--(5,+{sqrt(3)}/3-1.5)--(3,+{sqrt(147)}/3-1.5)--(1,+{sqrt(3)}/3-1.5);\draw[thick] (2,+{sqrt(48)}/3-1.5)--(3,+{sqrt(3)}/3-1.5)--(4,+{sqrt(48)}/3-1.5)--(2,+{sqrt(48)}/3-1.5); \draw[thick] (0,-1.5)--(1,+{sqrt(3)}/3-1.5)--(1,+{sqrt(3)}-1.5)--(2,+{sqrt(48)}/3-1.5)--(2,+{sqrt(12)}-1.5)--(3,+{sqrt(147)}/3-1.5)--(3,+{sqrt(27)}-1.5)--(3,+{sqrt(147)}/3-1.5)--(4,+{sqrt(12)}-1.5)--(4,+{sqrt(48)}/3-1.5)--(5,+{sqrt(3)}-1.5)--(5,+{sqrt(3)}/3-1.5)--(6,-1.5)--(5,+{sqrt(3)}/3-1.5)--(4,-1.5)--(3,+{sqrt(3)}/3-1.5)--(2,-1.5)--(1,+{sqrt(3)}/3-1.5);\draw [dashed] (1,+{sqrt(3)}-1.5)--(3,+{sqrt(3)}/3-1.5)--(3,+{sqrt(3)}-1.5)--(2,+{sqrt(48)}/3-1.5);\draw[dashed] (2,+{sqrt(12)}-1.5)--(4,+{sqrt(48)}/3-1.5)--(3,+{sqrt(3)}-1.5)--(5,+{sqrt(3)}/3-1.5);\draw [dashed] (2,-1.5)--(3,+{sqrt(3)}-1.5);\draw [dashed] (2,-1.5)--(1,+{sqrt(3)}-1.5);
\draw [thick] (-2,-{sqrt(12)})--(2,-{sqrt(108)})--(-6,-{sqrt(108)}); \draw [thick] (-6,-{sqrt(108)})--(-2,-{sqrt(12)})--(-2,-{sqrt(192)}/3)--(-3,-{sqrt(27)})--(-3,-{sqrt(363)}/3)--(-4,-{sqrt(48)})--(-4,-{sqrt(588)}/3)--(-5,-{sqrt(75)})--(-5,-{sqrt(867)}/3); \draw [dashed] (-5,-{sqrt(75)})--(-4,-{sqrt(108)})--(-3,-{sqrt(75)})--(-2,-{sqrt(108)})--(-1,-{sqrt(75)})--(0,-{sqrt(108)})--(1,-{sqrt(75)})--(-5,-{sqrt(75)})--(-4,-{sqrt(48)})--(-3,-{sqrt(75)})--(-2,-{sqrt(48)})--(-1,-{sqrt(75)})--(0,-{sqrt(48)})--(-4,-{sqrt(48)})--(-3,-{sqrt(27)})--(-2,-{sqrt(48)})--(-1,-{sqrt(27)})--(-3,-{sqrt(27)}); 
\draw[thick] (-2,-{sqrt(192)}/3)--(1,-{sqrt(867)}/3)--(-5,-{sqrt(867)}/3); \draw [thick] (-5,-{sqrt(867)}/3)--(-4,-{sqrt(588)}/3)--(-3,-{sqrt(867)}/3)--(-2,-{sqrt(588)}/3)--(-1,-{sqrt(867)}/3)--(0,-{sqrt(588)}/3)--(-4,-{sqrt(588)}/3)--(-3,-{sqrt(363)}/3)--(-2,-{sqrt(588)}/3)--(-1,-{sqrt(363)}/3)--(-3,-{sqrt(363)}/3)--(-2,-{sqrt(192)}/3);  \draw [thick] (-6,-{sqrt(108)})--(-5,-{sqrt(867)}/3)--(-4,-{sqrt(108)})--(-3,-{sqrt(867)}/3)--(-2,-{sqrt(108)})--(-1,-{sqrt(867)}/3)--(0,-{sqrt(108)})--(1,-{sqrt(867)}/3)--(2,-{sqrt(108)})--(1,-{sqrt(867)}/3)--(1,-{sqrt(75)})--(0,-{sqrt(588)}/3)--(0,-{sqrt(48)})--(-1,-{sqrt(363)}/3)--(-1,-{sqrt(27)})--(-2,-{sqrt(192)}/3)--(-2,-{sqrt(12)});
\draw [dashed] (-5,-{sqrt(75)})--(-3,-{sqrt(867)}/3)--(-3,-{sqrt(75)})--(-4,-{sqrt(588)}/3)--(-4,-{sqrt(48)})--(-2,-{sqrt(588)}/3)--(-2,-{sqrt(48)})--(-3,-{sqrt(363)}/3)--(-3,-{sqrt(27)})--(-1,-{sqrt(363)}/3)--(-2,-{sqrt(48)})--(0,-{sqrt(588)}/3)--(-1,-{sqrt(75)})--(-2,-{sqrt(588)}/3)--(-3,-{sqrt(75)})--(-1,-{sqrt(867)}/3)--(-1,-{sqrt(75)})--(1,-{sqrt(867)/3});
\node at (-8.2,+{sqrt(3)}/3-1.5) {$v_{0,0,0}$};
\node at (-10,-1.8) {$v_{1,0,0}$};\node at (-8,-1.8) {$v_{1,1,0}$}; \node at (-9,+{sqrt(3)}-1.2) {$v_{1,1,1}$}; 
\node at(-3.2,+{sqrt(12)}-1.5) {$v_{2,2,2}$}; \node at (-5.7,+{sqrt(3)}-1.5) {$v_{2,1,1}$}; \node at (-6,-1.8) {$v_{2,0,0}$}; \node at (-4,-1.8) {$v_{2,1,0}$}; \node at (-2,-1.8) {$v_{2,2,0}$}; \node at (-2.2,+{sqrt(3)}-1.5) {$v_{2,2,1}$};  \node at (-5.7,+{sqrt(3)}/3-1.5) {$v_{1,0,0}$}; \node at (-2.2,+{sqrt(3)}/3-1.5) {$v_{1,1,0}$}; \node at (-4.7,+{sqrt(48)}/3-1.5) {$v_{1,1,1}$}; 
\node at (0,-1.8) {$v_{3,0,0}$}; \node at (2,-1.8) {$v_{3,1,0}$};\node at (4,-1.8) {$v_{3,2,0}$};\node at (6,-1.8) {$v_{3,3,0}$};\node at (5.7,+{sqrt(3)}-1.5) {$v_{3,3,1}$}; \node at (4.7,+{sqrt(12)}-1.5) {$v_{3,3,2}$}; \node at (3.7,+{sqrt(27)}-1.5) {$v_{3,3,3}$}; \node at (0.3, +{sqrt(3)}-1.5) {$v_{3,1,1}$}; \node at (1.3,+{sqrt(12)}-1.5) {$v_{3,2,2}$};\node at(0.3,+{sqrt(3)}/3-1.5) {$v_{2,0,0}$}; \node at (1.2,+{sqrt(48)}/3-1.5) {$v_{2,1,1}$}; \node at (2.2,+{sqrt(147)}/3-1.5) {$v_{2,2,2}$}; \node at  (3,-1.2) {$v_{2,1,0}$}; \node at (5.7,+{sqrt(3)}/3-1.5) {$v_{2,2,0}$}; \node at (4.7,+{sqrt(48)}/3-1.5) {$v_{2,2,1}$};\node at (3.7,0.5) {$v_{3,2,1}$}; 
\node at (-1.3,-{sqrt(12)}) {$v_{4,4,4}$}; \node at (-0.3,-{sqrt(27)}) {$v_{4,4,3}$}; \node at (0.7,-{sqrt(48)}) {$v_{4,4,2}$}; \node at (1.7,-{sqrt(75)}) {$v_{4,4,1}$}; \node at (2.7,-{sqrt(108)}) {$v_{4,4,0}$}; \node at (0,-{sqrt(108)}-0.2) {$v_{4,3,0}$}; \node at (-2,-{sqrt(108)}-0.2) {$v_{4,2,0}$}; \node at (-4,-{sqrt(108)}-0.2) {$v_{4,1,0}$}; \node at (-6,-{sqrt(108)}-0.2) {$v_{4,0,0}$}; \node at (-5.7,-{sqrt(75)}) {$v_{4,1,1}$}; \node at (-4.7,-{sqrt(48)}) {$v_{4,2,2}$}; \node at (-3.7,-{sqrt(27)}) {$v_{4,3,3}$}; \node at (-2.7,-{sqrt(192)}/3) {$v_{3,3,3}$}; \node at (-3.7,-{sqrt(363)}/3) {$v_{3,2,2}$}; \node at (-4.7,-{sqrt(588)}/3) {$v_{3,1,1}$}; \node  at (-5.7,-{sqrt(867)}/3) {$v_{3,0,0}$}; \node at (-2.2,-{sqrt(867)}/3+0.2) {$v_{3,1,0}$}; \node at (-0.2,-{sqrt(867)}/3+0.2) {$v_{3,2,0}$}; \node at (1.7,-{sqrt(867)}/3) {$v_{3,3,0}$}; \node at (0.7,-{sqrt(588)}/3) {$v_{3,3,1}$}; \node at (-0.3,-{sqrt(363)}/3) {$v_{3,3,2}$}; \node at (-1.2,-{sqrt(588)}/3+0.2) {$v_{3,2,1}$};\node at (-3,-{sqrt(75)}+0.3) {$v_{4,2,1}$};\node at (-1,-{sqrt(75)}+0.3) {$v_{4,3,1}$}; \node at (-2,-{sqrt(48)}+0.3) {$v_{4,3,2}$};

\node at (4,-{sqrt(75)}+2) {\tiny{Color of vertices}};\node at (4.4,-{sqrt(75)}+1.5) {\tiny$:0$}; \node at (4.4,-{sqrt(75)}+1) {\tiny $:1$}; \node at (4.4,-{sqrt(75)}+0.5) {\tiny $:2$};\node at (4.4,-{sqrt(75)}) {\tiny $:3$}; 

\end{tikzpicture}
\end{center}
\caption{A fundamental domain for $\Gamma\backslash\mathcal{B}_4$}\label{figure}
\end{figure}
\begin{table}[h]
\centering
\begin{tabular}{|c|c|c|c|}
\noalign{\smallskip}\noalign{\smallskip}\hline
& \multicolumn{1}{|c|}{$i=1$}& \multicolumn{1}{|c|}{$i=2$} & \multicolumn{1}{|c|}{$i=3$} \\
\hline
\multicolumn{1}{|c|}{$V_{0,0,0}^i$} &$v_{1,0,0}$&$v_{1,1,0}$ &$v_{1,1,1}$ \\
\hline
\multicolumn{1}{|c|}{$V_{\ell,0,0}^i$}& $v_{\ell,1,0},v_{\ell+1,0,0}$&$v_{\ell,1,1},v_{\ell+1,1,0}$ & $v_{\ell-1,0,0},v_{\ell+1,1,1}$ \\
\hline
\multicolumn{1}{|c|}{$V_{\ell,\ell,0}^i$}  & $v_{\ell,\ell,1},v_{\ell+1,\ell,0}$ & $v_{\ell-1,\ell-1,0},v_{\ell+1,\ell+1,0},v_{\ell+1,\ell,1}$ & $v_{\ell,\ell-1,0},v_{\ell+1,\ell+1,1}$ \\
\hline
\multicolumn{1}{|c|}{$V_{\ell,\ell,\ell}^i$}  & $v_{\ell-1,\ell-1,\ell-1},v_{\ell+1,\ell,\ell}$ & $v_{\ell,\ell-1,\ell-1},v_{\ell+1,\ell+1,\ell}$ & $v_{\ell,\ell,\ell-1},v_{\ell+1,\ell+1,\ell+1}$ \\ \hline

\multicolumn{1}{|c|}{\multirow{2}{*}{$V_{\ell,m,0}^i$}}  & $v_{\ell,m,1},v_{\ell,m+1,0},$ & $v_{\ell-1,m-1,0},v_{\ell,m+1,1},$ & $v_{\ell-1,m,0},v_{\ell,m-1,0},$ \\ 
 & $v_{\ell+1,m,0}$ & $v_{\ell+1,m,1},v_{\ell+1,m+1,0}$ & $v_{\ell+1,m+1,1}$ \\ \hline
 \multicolumn{1}{|c|}{\multirow{2}{*}{$V_{\ell,m,m}^i$}}  & $v_{\ell-1,m-1,m-1},v_{\ell,m+1,m},$ & $v_{\ell-1,m,m-1},v_{\ell,m-1,m-1},$ & $v_{\ell-1,m,m},v_{\ell,m,m-1},$ \\ 
 
 &$v_{\ell+1,m,m}$&$v_{\ell,m+1,m+1},v_{\ell+1m+1,m}$&$v_{\ell+1,m+1,m+1}$\\ \hline

 \multicolumn{1}{|c|}{\multirow{2}{*}{$V_{\ell,\ell,m}^i$}}  & $v_{\ell-1,\ell-1,m-1},v_{\ell,\ell,m+1},$ & $v_{\ell-1,\ell-1,m},v_{\ell,\ell-1,m-1},$ & $v_{\ell,\ell-1,m},v_{\ell,\ell,m-1},$ \\ 
 
 &$v_{\ell+1,\ell,m}$&$v_{\ell+1,\ell,m+1},v_{\ell+1,\ell+1,m}$&$v_{\ell+1,\ell+1,m+1}$\\ \hline
  \multicolumn{1}{|c|}{\multirow{3}{*}{$V_{\ell,m,n}^i$}}  & $v_{\ell-1,m-1,n-1},v_{\ell,m,n+1},$ & $v_{\ell-1,m-1,n},v_{\ell-1,m,n-1},$ & $v_{\ell-1,m,n},v_{\ell,m-1,n},$ \\ 
  &$v_{\ell,m+1,n},v_{\ell+1,m,n}$&$v_{\ell,m-1,n-1},v_{\ell,m+1,n+1}$&$v_{\ell,m,n-1},v_{\ell+1,m+1,n+1}$\\
  &&$v_{\ell+1,m,n+1},v_{\ell+1,m+1,n}$&\\ \hline
\end{tabular}
\vspace{1.5em}
\caption{The neighborhoods of the vertices}\label{t1}
\end{table}


\section{Calculation of natural weights}\label{sec:4}

 Let $\Gamma_{\ell,m,n}$ be the stabilizer of a vertex $x_{\ell,m,n}W$ in $\mathcal{B}_4$. We consider the weight function $w$ by 
 $$w(v_{\ell,m,n}):=\frac{(q-1)^3q^6}{|\Gamma_{\ell,m,n}|}\text{ and } w(e):=\frac{(q-1)^3q^6}{|\Gamma_{\ell,m,n}\cap \Gamma_{\ell',m',n'}|},$$
 where $e=(v_{\ell,m,n},v_{\ell',m',n'})$ and $|S|$ is the cardinality of a set $S$. Note that the numerator constant does not affect the space $L^2_w(\Gamma\backslash\mathcal{B}_4)$ or the operators $A_{w,i}$. In this section, we compute the values of the weight function $w$.
 \begin{lem}\label{lem:4.2}The cardinality of $\Gamma_{0,0,0}$ is  $(q-1)^3(q^3+q^2+q+1)(q^2+q+1)(q+1)q^6.$
For any $m>0$,  
 \begin{equation}\label{eq:4.1}
 \begin{split}
&|\Gamma_{\ell,0,0}|=(q-1)^3(q^2+q+1)(q+1)q^{3\ell+6}\\
&|\Gamma_{\ell,\ell,0}|=(q-1)^3(q+1)^2q^{4\ell+6}\\
 &|\Gamma_{\ell,\ell,\ell}|=(q-1)^3(q^2+q+1)(q+1)q^{3\ell+6}.
 \end{split}
 \end{equation}
 \end{lem}
 \begin{proof} Since $\gamma x_{\ell,m,n}W=x_{\ell,m,n}W$ if and only if $\gamma\in \Gamma_{\ell,m,n}$, we have 
 $$\Gamma_{\ell,m,n}=\{\gamma\in \Gamma:x_{\ell,m,n}^{-1}\gamma x_{\ell,m,n}\in W\}.$$
 Obviously, $\Gamma_{0,0,0}=\Gamma\cap PGL(4,\mathcal{O}_F)=PGL(4,\mathbb{F}_q)$ and
  $$|\Gamma_{0,0,0}|=\frac{(q^4-1)(q^4-q)(q^4-q^2)(q^4-q^3)}{q-1}=(q-1)^3(q^3+q^2+q+1)(q^2+q+1)(q+1)q^6.$$
 Let $P^n(t)$ be the space of polynomials of degree less and equal to $n$. For any $m>0$,
 \begin{equation}\nonumber
 \begin{split}
& \Gamma_{\ell,0,0}=\left\{A=(a_{ij})\in \Gamma:\substack{a_{21}=a_{31}=a_{41}=0, a_{12},a_{13},a_{14}\in P^\ell(t),\\a_{11},a_{22},a_{23},a_{24},a_{32},a_{33},a_{34},a_{42},a_{43},a_{44}\in \mathbb{F}_q} \right\}/\{\lambda I:\lambda\in \mathbb{F}_q^\times\}\\
  &\Gamma_{\ell,\ell,0}=\left\{A=(a_{ij})\in \Gamma:\substack{a_{31}=a_{32}=a_{41}=a_{42}=0,a_{13},a_{14},a_{23},a_{24}\in P^\ell(t)\\a_{11},a_{12},a_{21}a_{22},a_{33},a_{34},a_{43},a_{44}\in \mathbb{F}_q} \right\}/\{\lambda I:\lambda\in \mathbb{F}_q\}\\
   &\Gamma_{\ell,\ell,\ell}=\left\{A=(a_{ij})\in \Gamma:\substack{a_{41}=a_{42}=a_{43}=0, a_{14},a_{24},a_{34}\in P^\ell(t),\\a_{11},a_{12},a_{13},a_{21},a_{22},a_{23},a_{31},a_{32},a_{33},a_{44}\in \mathbb{F}_q} \right\}/\{\lambda I:\lambda\in \mathbb{F}_q^\times\}.
 \end{split}
 \end{equation}
 Since for any $A\in \Gamma_{\ell,0,0}$, $a_{11}$ is nonzero and the vectors $(a_{22},a_{32},a_{42})$, $(a_{23},a_{33},a_{43})$ and $(a_{24},a_{34},a_{44})$ are linearly independent,
 $$|\Gamma_{\ell,0,0}|=\frac{(q-1)(q^3-1)(q^3-q)(q^3-q^2)q^{3\ell+3}}{q-1}=(q-1)^3(q^2+q+1)(q+1)q^{3\ell+6}.$$
 Similarly,  $|\Gamma_{\ell,\ell,\ell}|=(q-1)^3(q^2+q+1)(q+1)q^{3\ell+6}.$
 
 For any $A\in \Gamma_{\ell,\ell,0}$, two matrices $\begin{pmatrix} a_{11}&a_{12}\\a_{21}&a_{22}\end{pmatrix}$ and $\begin{pmatrix} a_{33}&a_{34}\\a_{43}&a_{44}\end{pmatrix}$ are invertible. Using this,
 $$|\Gamma_{\ell,\ell,0}|=\frac{(q^2-1)^2(q^2-q)^2q^{4\ell+4}}{q-1}=(q-1)^3(q+1)^2q^{4\ell+6}.$$
 \end{proof}
 \begin{lem}\label{lem:4.3}For any $\ell>m>0,$
 \begin{equation}\label{eq:4.2}
 \begin{split}
 &|\Gamma_{\ell,m,0}|=(q-1)^3(q+1)q^{3\ell+m+6}\\
 &|\Gamma_{\ell,m,m}| =(q-1)^3(q+1)q^{3\ell+6}\\
 &|\Gamma_{\ell,\ell,m}|=(q-1)^3(q+1)q^{4\ell-m+6}.
 \end{split}
 \end{equation}
 \end{lem}
 \begin{proof}
 Similar to the proof of Lemma \ref{lem:4.2}, 
 \begin{equation}\nonumber
 \begin{split}
 & \Gamma_{\ell,m,0}=\left\{A=(a_{ij})\in \Gamma:\substack{a_{21}=a_{31}=a_{32}=a_{41}=a_{42}=0, a_{12}\in P^{\ell-m}(t),a_{13},a_{14}\in P^\ell(t),\\ a_{23},a_{24}\in P^m(t), a_{11},a_{22},a_{33},a_{34},a_{43},a_{44}\in \mathbb{F}_q} \right\}/\{\lambda I:\lambda\in \mathbb{F}_q^\times\}\\
& \Gamma_{\ell,m,m}=\left\{A=(a_{ij})\in \Gamma:\substack{a_{21}=a_{31}=a_{41}=a_{42}=a_{43}=0, a_{12},a_{13}\in P^{\ell-m}(t),a_{14}\in P^\ell(t),\\ a_{24},a_{34}\in P^{m}(t),a_{11},a_{22},a_{23},a_{32},a_{33},a_{44}\in \mathbb{F}_q} \right\}/\{\lambda I:\lambda\in \mathbb{F}_q^\times\}\\
& \Gamma_{\ell,\ell,m}=\left\{A=(a_{ij})\in \Gamma:\substack{a_{31}=a_{32}=a_{41}=a_{42}=a_{43}=0,a_{13},a_{23}\in P^{\ell-m}(t),\\a_{14},a_{24}\in P^\ell(t),a_{34}\in P^m(t),a_{11},a_{12},a_{21},a_{22},a_{33},a_{44}\in \mathbb{F}_q} \right\}/\{\lambda I:\lambda\in \mathbb{F}_q^\times\}.
 \end{split}
 \end{equation}
Following the method in the proof of Lemma \ref{lem:4.2}, we have \eqref{eq:4.2}.
 \end{proof}
 \begin{lem}\label{lem:4.4}For any $\ell>m>n>0,$
 \begin{equation}\label{eq:4.3}
 |\Gamma_{\ell,m,n}|=(q-1)^3q^{3\ell+m-n+6}
 \end{equation}
 \end{lem}
 \begin{proof}Since 
 \begin{equation}\nonumber
  \Gamma_{\ell,m,n}=\left\{A=(a_{ij})\in \Gamma:\substack{a_{21}=a_{31}=a_{32}=a_{41}=a_{42}=a_{43}=0, a_{12}\in P^{\ell-m}(t),\\a_{13}\in P^{\ell-n}(t),a_{14}\in P^\ell(t),a_{23}\in P^{m-n},a_{24}\in P^m(t) ,a_{34}\in P^n(t),\\a_{11},a_{22},a_{33},a_{44}\in \mathbb{F}_q} \right\}/\{\lambda I:\lambda\in \mathbb{F}_q^\times\},
 \end{equation}
 the equation \eqref{eq:4.3} holds.
 \end{proof}
 \begin{lem}\label{lem:4.5}
For any $f\in L^2_w(X),$
\begin{equation}\label{eq:4.4}
\begin{split}
&A_{w,1}f(v_{0,0,0})=(q^3+q^2+q+1)f(v_{1,0,0})\\
&A_{w,2}f(v_{0,0,0})=(q^4+q^3+2q^2+q+1)f(v_{1,1,0})\\
&A_{w,3}f(v_{0,0,0})=(q^3+q^2+q+1)f(v_{1,1,1})
\end{split}
\end{equation}
 \end{lem}
 \begin{proof}
It follows from the proof of Lemma \ref{lem:4.2} that 
\begin{equation}\nonumber
\begin{split}
&|\Gamma_{0,0,0}\cap\Gamma_{1,0,0}|=(q-1)^3(q^2+q+1)(q+1)q^6\\
&|\Gamma_{0,0,0}\cap\Gamma_{1,1,0}|=(q-1)^3(q+1)^2q^6\\
&|\Gamma_{0,0,0}\cap\Gamma_{1,1,1}|=(q-1)^3(q^2+q+1)(q+1)q^6.
\end{split}
\end{equation}
This shows that $\frac{w(v_{0,0,0},v_{1,0,0})}{w(v_{0,0,0})}=q^3+q^2+q+1$, $\frac{w(v_{0,0,0},v_{1,1,0})}{w(v_{0,0,0})}=q^4+q^3+2q^2+q+1$ and $\frac{w(v_{0,0,0},v_{1,1,1})}{w(v_{0,0,0})}=q^3+q^2+q+1$. Thus we have \eqref{eq:4.4}.
\end{proof}
For any $\ell>0$, we have the following lemmata.
\begin{lem}\label{lem:4.6}For any $f\in L^2_w(X)$ and $\ell>0$,
\begin{equation}\label{eq:4.5}
\begin{split}
&A_{w,1}f(v_{\ell,0,0})=(q^3+q^2+q)f(v_{\ell,1,0})+f(v_{\ell+1,0,0})\\
&A_{w,2}f(v_{\ell,0,0})=(q^4+q^3+q^2)f(v_{\ell,1,1})+(q^2+q+1)f(v_{\ell+1,1,0})\\
&A_{w,3}f(v_{\ell,0,0})=q^3f(v_{\ell-1,0,0})+(q^2+q+1)f(v_{\ell+1,1,1}).
\end{split}
\end{equation}
\end{lem}
\begin{proof}It follows from the proof of Lemma \ref{lem:4.2} and Lemma \ref{lem:4.3} that 
\begin{equation}\nonumber
\begin{split}
&|\Gamma_{\ell,0,0}\cap\Gamma_{\ell,1,0}|=(q-1)^3(q+1)q^{3\ell+5}\\
&|\Gamma_{\ell,0,0}\cap\Gamma_{\ell,1,1}|=(q-1)^3(q+1)q^{3\ell+4}\\
&|\Gamma_{\ell,0,0}\cap\Gamma_{\ell+1,0,0}|=(q-1)^3(q^2+q+1)(q+1)q^{3\ell+6}\\
&|\Gamma_{\ell,0,0}\cap\Gamma_{\ell+1,1,0}|=|\Gamma_{\ell,0,0}\cap\Gamma_{\ell+1,1,1}|=(q-1)^3(q+1)q^{3\ell+6}.
\end{split}
\end{equation}
Similar to the proof of Lemma \ref{lem:4.5}, we have \eqref{eq:4.5}.
\end{proof}
\begin{lem}\label{lem:4.7}
For any $f\in L^2_w(X)$ and $\ell>0$,
\begin{equation}\label{eq:4.6}
\begin{split}
&A_{w,1}f(v_{\ell,\ell,0})=(q^3+q^2)f(v_{\ell,\ell,1})+(q+1)f(v_{\ell+1,\ell,0})\\
&A_{w,2}f(v_{\ell,\ell,0})=q^4f(v_{\ell-1,\ell-1,0})+q(q+1)^2f(v_{\ell+1,\ell,1})+f(v_{\ell+1,\ell+1,0})\\
&A_{w,3}f(v_{\ell,\ell,0})=(q^3+q^2)f(v_{\ell,\ell-1,0})+(q+1)f(v_{\ell+1,\ell+1,1}).
\end{split}
\end{equation}
\end{lem}
\begin{proof}It follows from Lemma \ref{lem:4.2} and Lemma \ref{lem:4.3} that  
\begin{equation}\nonumber
\begin{split}
&|\Gamma_{\ell,\ell,0}\cap\Gamma_{\ell,\ell-1,0}|=|\Gamma_{\ell,\ell,0}\cap\Gamma_{\ell,\ell,1}|=(q-1)^3(q+1)q^{4\ell+4}\\
&|\Gamma_{\ell,\ell,0}\cap\Gamma_{\ell+1,\ell,0}|=|\Gamma_{\ell,\ell,0}\cap\Gamma_{\ell+1,\ell+1,1}|=(q-1)^3(q+1)q^{4\ell+6}\\
&|\Gamma_{\ell,\ell,0}\cap\Gamma_{\ell+1,\ell,1}|=(q-1)^3q^{4\ell+5}\\
&|\Gamma_{\ell,\ell,0}\cap\Gamma_{\ell+1,\ell+1,0}|=(q-1)^3(q+1)^2q^{4\ell+6}.
\end{split}
\end{equation}
Similar to the proof of Lemma \ref{lem:4.5}, we have \eqref{eq:4.6}.
\end{proof}
\begin{lem}\label{lem:4.8}For any $f\in L_w^2(X)$ and $\ell>0$,
\begin{equation}\label{eq:4.7}
\begin{split}
&A_{w,1}f(v_{\ell,\ell,\ell})=q^3f(v_{\ell-1,\ell-1,\ell-1})+(q^2+q+1)f(v_{\ell+1,\ell,\ell})\\
&A_{w,2}f(v_{\ell,\ell,\ell})=(q^4+q^3+q^2)f(v_{\ell,\ell-1,\ell-1})+(q^2+q+1)f(v_{\ell+1,\ell+1,\ell})\\
&A_{w,3}f(v_{\ell,\ell,\ell})=(q^3+q^2+q)f(v_{\ell,\ell,\ell-1})+f(v_{\ell+1,\ell+1,\ell+1}).
\end{split}
\end{equation}
\end{lem}
\begin{proof}
Using Lemma \ref{lem:4.2} and Lemma \ref{lem:4.3}, we have
\begin{equation}\nonumber
\begin{split}
&|\Gamma_{\ell,\ell,\ell}\cap\Gamma_{\ell,\ell-1,\ell-1}|=(q-1)^3(q+1)q^{3\ell+4}\\
&|\Gamma_{\ell,\ell,\ell}\cap\Gamma_{\ell,\ell,\ell-1}|=(q-1)^3(q+1)q^{3\ell+5}\\
&|\Gamma_{\ell,\ell,\ell}\cap\Gamma_{\ell+1,\ell,\ell}|=|\Gamma_{\ell,\ell,\ell}\cap\Gamma_{\ell+1,\ell+1,\ell}|=(q-1)^3(q+1)q^{3\ell+6}\\
&|\Gamma_{\ell,\ell,\ell}\cap\Gamma_{\ell+1,\ell+1,\ell+1}|=(q-1)^3(q^2+q+1)(q+1)q^{3\ell+6}.
\end{split}
\end{equation}
This shows \eqref{eq:4.7}.
\end{proof} 
For any $\ell>m>0$, we have the following lemmata.
\begin{lem}\label{lem:4.9} For any $f\in L^2_w(X)$ and $\ell>m$,
\begin{equation}\label{eq:4,8}
\begin{split}
A_{w,1}f(v_{\ell,m,0})=&\,(q^3+q^2)f(v_{\ell,m,1})+qf(v_{\ell,m+1,0})+f(v_{\ell+1,m,0})\\
A_{w,2}f(v_{\ell,m,0})=&\,q^4f(v_{\ell-1,m-1,0})+(q^3+q^2)f(v_{\ell,m+1,1})\\
&+(q^2+q)f(v_{\ell+1,m,1})+f(v_{\ell+1,m+1,0})\\
A_{w,3}f(v_{\ell,m,0})=&\,q^3f(v_{\ell-1,m,0})+q^2f(v_{\ell,m-1,0})+(q+1)f(v_{\ell+1,m+1,1}).
\end{split}
\end{equation}
\end{lem}
\begin{proof}
Using Lemma \ref{lem:4.2}, Lemma \ref{lem:4.3} and Lemma \ref{lem:4.4}, we have
\begin{equation}\nonumber
\begin{split}
&|\Gamma_{\ell,m,0}\cap\Gamma_{\ell,m,1}|=|\Gamma_{\ell,m,0}\cap\Gamma_{\ell,m+1,1}|=(q-1)^3q^{3\ell+m+4}\\
&|\Gamma_{\ell,m,0}\cap\Gamma_{\ell,m+1,0}|=(q-1)^3(q+1)q^{3\ell+m+5}\\
&|\Gamma_{\ell,m,0}\cap\Gamma_{\ell+1,m,0}|=|\Gamma_{\ell,m,0}\cap\Gamma_{\ell+1,m+1,0}|=(q-1)^3(q+1)q^{3\ell+m+6}\\
&|\Gamma_{\ell,m,0}\cap\Gamma_{\ell-1,m-1,0}|=(q-1)^3(q+1)q^{3\ell+m+2}\\
&|\Gamma_{\ell,m,0}\cap\Gamma_{\ell+1,m,1}|=(q-1)^3q^{3\ell+m+5}\\
&|\Gamma_{\ell,m,0}\cap\Gamma_{\ell+1,m+1,1}|=(q-1)^3q^{3\ell+m+6}.
\end{split}
\end{equation}
which implies \eqref{eq:4,8}.
\end{proof}
\begin{lem}\label{eq:4.8} For any $f\in L_w^2(X)$ and $\ell>m$,
\begin{equation}\label{eq:4.9}
\begin{split}
A_{w,1}f(v_{\ell,m,m})=&q^3f(v_{\ell-1,m-1,m-1})+(q^2+q)f(v_{\ell,m+1,m})+f(v_{\ell+1,m,m})\\
A_{w,2}f(v_{\ell,m,m})=&(q^4+q^3)f(v_{\ell-1,m,m-1})+q^2f(v_{\ell,m-1,m-1})\\
&+q^2f(v_{\ell,m+1,m+1})+(q+1)f(v_{\ell+1,m+1,m})\\
A_{w,3}f(v_{\ell,m,m})=&q^3f(v_{\ell-1,m,m})+(q^2+q)f(v_{\ell,m,m-1})+f(v_{\ell+1,m+1,m+1})
\end{split}
\end{equation}
\end{lem}
\begin{proof}From Lemma \ref{lem:4.2}, Lemma \ref{lem:4.3} and Lemma \ref{lem:4.4}, we get
\begin{equation}\nonumber
\begin{split}
&|\Gamma_{\ell,m,m}\cap\Gamma_{\ell-1,m-1,m-1}|=(q-1)^3(q+1)q^{3\ell+3}\\
&|\Gamma_{\ell,m,m}\cap\Gamma_{\ell-1,m,m-1}|=(q-1)^3q^{3\ell+3}\\
&|\Gamma_{\ell,m,m}\cap\Gamma_{\ell-1,m,m}|=(q-1)^3(q+1)q^{3\ell+3}\\
&|\Gamma_{\ell,m,m}\cap\Gamma_{\ell,m-1,m-1}|=(q-1)^3(q+1)q^{3\ell+4}\\
&|\Gamma_{\ell,m,m}\cap\Gamma_{\ell,m,m-1}|=|\Gamma_{\ell,m,m}\cap\Gamma_{\ell,m+1,m}|=(q-1)^3q^{3\ell+5}\\
&|\Gamma_{\ell,m,m}\cap\Gamma_{\ell+1,m,m}|=(q-1)^3(q+1)q^{3\ell+6}\\
&|\Gamma_{\ell,m,m}\cap\Gamma_{\ell+1,m+1,m}|=(q-1)^3q^{3\ell+6}.
\end{split}
\end{equation}
which yields \eqref{eq:4.9}.
\end{proof}
\begin{lem}For any $f\in L^2_w(X)$ and $\ell>m$,
\begin{equation}\label{eq:4.10}
\begin{split}
A_{w,1}f(v_{\ell,\ell,m})=&\,q^3f(v_{\ell-1,\ell-1,m-1})+q^2f(v_{\ell,\ell,m+1})+(q+1)f(v_{\ell+1,\ell,m})\\
A_{w,2}f(v_{\ell,\ell,m})=&\,q^4f(v_{\ell-1,\ell-1,m})+(q^3+q^2)f(v_{\ell,\ell-1,m-1})\\
&+(q^2+q)f(v_{\ell+1,\ell,m+1})+f(v_{\ell+1,\ell+1,m})\\
A_{w,3}f(v_{\ell,\ell,m})=&\,(q^3+q^2)f(v_{\ell,\ell-1,m})+qf(v_{\ell,\ell,m-1})+f(v_{\ell+1,\ell+1,m+1}).
\end{split}
\end{equation}
\end{lem}
\begin{proof}Using Lemma \ref{lem:4.2}, Lemma \ref{lem:4.3} and Lemma \ref{lem:4.4}, we have
\begin{equation}\nonumber
\begin{split}
&|\Gamma_{\ell,\ell,m}\cap\Gamma_{\ell-1,\ell-1,m-1}|=(q-1)^3(q+1)q^{4\ell-m+3}\\
&|\Gamma_{\ell,\ell,m}\cap\Gamma_{\ell-1,\ell-1,m}|=(q-1)^3(q+1)q^{4\ell-m+2}\\
&|\Gamma_{\ell,\ell,m}\cap\Gamma_{\ell,\ell-1,m-1}|=(q-1)^3q^{4\ell-m+4}\\
&|\Gamma_{\ell,\ell,m}\cap\Gamma_{\ell,\ell-1,m}|=(q-1)^3q^{4\ell-m+4}\\
&|\Gamma_{\ell,\ell,m}\cap\Gamma_{\ell,\ell,m+1}|=(q-1)^3(q+1)q^{4\ell-m+4}\\
&|\Gamma_{\ell,\ell,m}\cap\Gamma_{\ell+1,\ell,m}|=(q-1)^3q^{4\ell-m+6}\\
&|\Gamma_{\ell,\ell,m}\cap\Gamma_{\ell+1,\ell,m+1}|=(q-1)^3q^{4\ell-m+5}.
\end{split}
\end{equation}
Using this, we have \eqref{eq:4.10}.
\end{proof}
For any $\ell>m>n>0$, we have the following lemma.
\begin{lem}\label{lem:4.11}For any $f\in L^2_w(X)$ and $\ell>m>n>0$,
\begin{equation*}
\begin{split}
A_{w,1}f(v_{\ell,m,n})=&\,q^3f(v_{\ell-1,m-1,n-1})+q^2f(v_{\ell,m,n+1})+qf(v_{\ell,m+1,n})+f(v_{\ell+1,m,n})\\
A_{w,2}f(v_{\ell,m,n})=&\,q^4f(v_{\ell-1,m-1,n})+q^3f(v_{\ell-1,m,n-1})+q^2f(v_{\ell,m-1,n-1})\\
&+q^2f(v_{\ell,m+1,n+1})+qf(v_{\ell+1,m,n+1})+f(v_{\ell+1,m+1,n})\\
A_{w,3}f(v_{\ell,m,n})=&\,q^3f(v_{\ell-1,m,n})+q^2f(v_{\ell,m-1,n})+qf(v_{\ell,m,n-1})+f(v_{\ell+1,m+1,n+1}).
\end{split}
\end{equation*}
\end{lem}
\begin{proof}Using Lemma \ref{lem:4.4}, we have
\begin{equation}\nonumber
\begin{split}
&|\Gamma_{\ell,m,n}\cap\Gamma_{\ell-1,\ell-1,m-1}|=(q-1)^3(q+1)q^{3\ell+m-n+3}\\
&|\Gamma_{\ell,m,n}\cap\Gamma_{\ell-1,m-1,n}|=(q-1)^3(q+1)q^{3\ell+m-n+2}\\
&|\Gamma_{\ell,m,n}\cap\Gamma_{\ell-1,m,n-1}|=(q-1)^3q^{3\ell+m-n+3}\\
&|\Gamma_{\ell,m,n}\cap\Gamma_{\ell,m-1,n-1}|=(q-1)^3q^{3\ell+m-n+4}\\
&|\Gamma_{\ell,m,n}\cap\Gamma_{\ell,m,n+1}|=(q-1)^3(q+1)q^{3\ell+m-n+4}\\
&|\Gamma_{\ell,m,n}\cap\Gamma_{\ell,m+1,n}|=(q-1)^3q^{3\ell+m-n+5}\\
&|\Gamma_{\ell,m,n}\cap\Gamma_{\ell+1,m,n}|=(q-1)^3q^{3\ell+m-n+6}.
\end{split}
\end{equation}
which yields the Lemma.
\end{proof}


\section{Simultaneous eigenfunctions of the weighted adjacency operators}\label{sec:5}
In this section, we summarize the results about simultaneous eigenfunctions of the weighted adjacency operators $A_{w,1}, A_{w,2}, A_{w,3}$ in the space of arbitrary functions on $\Gamma\backslash\mathcal{B}_4^{(0)}$. Detailed calculation will be presented in Appendix~\ref{sec:appendix}. Among the simultaneous eigenvalues $(\lambda_1,\lambda_2,\lambda_3)$, the only tuples for which the corresponding eigenfunction $f$ satisfies Condition $(B)$ can be contained in the automorphic spectrum (see Section~\ref{sec:6}). 

 By Lemma \ref{lem:4.6}, the eigenfunction $f$ of the simultaneous eigenvalue $\lambda_i$ of $A_{w,i}$ satisfies 
\begin{equation}\nonumber
\begin{split}
&\lambda_1f(v_{\ell+1,0,0})=(q^3+q^2+q)f(v_{\ell+1,1,0})+f(v_{\ell+2,0,0})\\
&\lambda_2f(v_{\ell,0,0})=(q^4+q^3+q^2)f(v_{\ell,1,1})+(q^2+q+1)f(v_{\ell+1,1,0})\\
&\lambda_3f(v_{\ell-1,0,0})=q^3f(v_{\ell-2,0,0})+(q^2+q+1)f(v_{\ell,1,1}).
\end{split}
\end{equation}
Let $a_{\ell,m,n}=f(v_{\ell,m,n})/q^{\ell+m+n}.$ Using the above equations, we have
\begin{equation}\nonumber
\begin{split}
&{q}^{-1}\lambda_1a_{\ell+1,0,0}=(q^3+q^2+q)a_{\ell+1,1,0}+a_{\ell+2,0,0}\\
&{q}^{-1}\lambda_2a_{\ell,0,0}=(q^5+q^4+q^3)a_{\ell,1,1}+(q^3+q^2+q)a_{\ell+1,1,0}\\
&\lambda_3 a_{\ell-1,0,0}=q^2a_{\ell-2,0,0}+(q^5+q^4+q^3)a_{\ell,1,1}.
\end{split}
\end{equation}
This implies that 
\begin{equation}\label{char poly}
a_{\ell+2,0,0}-{q}^{-1}\lambda_1a_{\ell+1,0,0}+q^{-1}\lambda_2a_{\ell,0,0}-\lambda_3a_{\ell-1,0,0}+q^2a_{\ell-2,0,0}=0.
\end{equation}
Let $\textbf{z}=(z_1,z_2,z_3,z_4)$ be the solution of the characteristic equation \eqref{char poly}. The 4-tuple $\textbf{z}$ satisfies
\begin{equation}\label{eq:para}
\lambda_1=q\sqrt{q}\sigma_1(\textbf{z}),\quad \lambda_2=q^2\sigma_2(\textbf{z}),\quad \lambda_3=q\sqrt{q}\sigma_3(\textbf{z})\quad \sigma_4(\textbf{z})=1,
\end{equation}
where $$\sigma_k(z_1,\ldots,z_d)=\sum_{1\le i_1<\cdots<i_k\le d} z_{i_1}z_{k_2}\cdots z_{i_k}$$ is the $k$-th elementary symmetric polynomial. Since $A_{w,i}^*=A_{w,4-i}$ for all $i=1,2,3$, every automorphic simultaneous eigenvalue $(\lambda_1,\lambda_2,\lambda_3)$ is parametrized by $\textbf{z}$ in the set
$$S=\{\textbf{z}=(z_1,z_2,z_3,z_4)\in \mathbb{C}^4:\overline{\sigma_1(\textbf{z})}=\sigma_3(\textbf{z}),\sigma_2(\textbf{z})\in \mathbb{R},z_1z_2z_3z_4=1\}.$$
and the eigenfunction $f=f_\textbf{z}$ is also indexed by $\textbf{z}.$

We assume $z_i\neq z_j$ for any $i,j$ in this section.
From this, we obtain the general formula of $a_{\ell,0,0}$.   
  Similarly, $a_{\ell,m,0}$ satisfies the following equation
  $$ a_{\ell+2,m+1,0}+qa_{\ell+1,m+2,0}-\frac{\lambda_1}{q}a_{\ell+1,m+1,0}+\frac{\lambda_3}{q}a_{\ell,m,0}-qa_{\ell-1,m,0}-a_{\ell,m-1,0}=0,$$ 
which again yields the formula of $a_{\ell,m,0}$ for any $\ell\geq m> 0$.

Similarly, for any $\ell\geq m\geq n>0$, we have 
$$\lambda_1a_{\ell,m,n}=a_{\ell-1,m-1,n-1}+q^3a_{\ell,m,n+1}+q^2a_{\ell,m+1,n}+qa_{\ell+1,m,n}.$$ 
This enables us to find $a_{\ell,m,n}$ and hence the simultaneous eigenfunction $f$. It requires Lemmata in Section \ref{sec:4} and complicated calculations to obtain the above equations. For more detail, we refer to see Appendix~\ref{sec:appendix}. 

To express the eigenfunctions, we introduce constants $C_{ijk}$ as follows. Let $$D=(q+1)(q^2+q+1)(q^3+q^2+q+1)$$ and $$C_{ijk}=\frac{(z_i-qz_k)(z_i-qz_j)(z_i-qz_\ell)(z_j-qz_k)(z_j-qz_\ell)(z_k-qz_\ell)}{D(z_i-z_k)(z_i-z_j)(z_i-z_\ell)(z_j-z_k)(z_j-z_\ell)(z_k-z_\ell)}.$$
For example,
$$C_{123}=\frac{(z_1-qz_2)(z_1-qz_3)(z_1-qz_4)(z_2-qz_3)(z_2-qz_4)(z_3-qz_4)}{D(z_1-z_2)(z_1-z_3)(z_1-z_4)(z_2-z_3)(z_2-z_4)(z_3-z_4)}.$$
The main result of this section is: 

\begin{thm}\label{thm:eigenftn}
 Let $(\lambda_1,\lambda_2,\lambda_3)$ be an element of $\mathbb{C}^3$ indexed by the equation~(\ref{eq:para}) with respect to $(z_1,z_2,z_3,z_4)\in S$. If $f=f_{\textbf{z}}$ is the corresponding simultaneous eigenfunction satisfying $A_{w,i}f=\lambda_if$, then  
$$f_{\textbf{z}}(v_{\ell,m,n})=\sum_{\tau\in S_4}C_{\tau(1)\tau(2)\tau(3)}\sqrt{q}^{3\ell+m-n}z_{\tau(1)}^\ell z_{\tau(2)}^m z_{\tau(3)}^n.$$
\end{thm}
Rigorous proof of the above theorem will be given in Appendix~\ref{sec:appendix}. Appendix \ref{sec:appendix} also deals with the case when the characteristic equation \eqref{char poly} has repeated roots.


\section{Automorphic spectrum of the weighted adjacency opartors}\label{sec:6}
In this section, we investigate the automorphic spectrum of the weighted adjacency operators $A_{w,1}, A_{w,2}, A_{w,3}$ on $L^2_w(\Gamma\backslash\mathcal{B}_4)$. We will prove that there are no discrete spectrum except trivial eigenvalues, and present all the simultaneous automorphic spectrum of $A_{w,i}$ for $i=1,2,3$.

We recall that the set $\mathcal{S}^4$ is the subset of $\mathbb{C}^3$ containing $(\lambda_1,\lambda_2,\lambda_3)$ given by \eqref{eq:para} for $|z_1|=|z_2|=|z_3|=|z_4|=z_1z_2z_3z_4=1$. 

Meanwhile, $(\lambda_1,\lambda_2,\lambda_3)$ is in the automorphic simultaneous spectrum of $A_{w,1},A_{w,2},A_{w,3}$ on $L^2_w(\Gamma\backslash\mathcal{B}_4)$ if and only if there exists a sequence $\{f_n\}$ of $L^2_w(\Gamma\backslash\mathcal{B}_4)$ with $\|f_n\|_{2,w}=1$ such that for any $i=1,2,3$,
   $$\lim_{n\rightarrow \infty} \|A_{w,i}f_n-\lambda_if_n\|_{2,w}=0.$$
\subsection{From the condition $(A)$ to the condition $(B)$}\label{sec:6.1} In this section, we show that the eigenfunction $f_{\textbf{z}}$ indexed by a 4-tuple $\textbf{z}$ of the form $(z_1,z_2,z_3,z_3)$ satisfies the condition $(B)$ whenever the condition $(A)$ holds. 
Since the condition $(B)$ holds when $|z_i|=1$ for any $i$ and $z_i\neq z_j$ for any $i,j$, it is enough to consider the case when $|z_i|>1$ for some $i$.

The following lemma reduces the number of possibilities that we need to consider the cases when $|z_i|>1$ for some $i$ and $z_j=z_k$ for some $j,k$.

\begin{lem}\label{lem:6.6} If a point $\textbf{z}=(z_1,z_2,z_3,z_4)\in S$ satisfies $|z_1|=|z_2|=|z_3|$, then $|z_i|=1$ for any $i$.
\end{lem} 
\begin{proof}Let $z_i=ae^{i\theta_i}$ for $i=1,2,3$ and $z_4=\frac{1}{a^3}e^{-i(\theta_1+\theta_2+\theta_3)}$. Since $\overline{\sigma_1(\textbf{z})}=\sigma_3(\textbf{z})$,
$$\biggl(a-\frac{1}{a}\biggr)\biggl|e^{-i\theta_1}+e^{-i\theta_2}+e^{-i\theta}\biggr|=\biggl(a^3-\frac{1}{a^3}\biggr).$$
This holds only if $a=1$ and $e^{-i\theta_1}=e^{-i\theta_2}=e^{-i\theta_3}$. This completes the proof.
\end{proof}
 By Lemma \ref{lem:6.6}, to prove Lemma \ref{lem:2-d}, it is enough to consider either the cases when $|z_1|>|z_2|$ or when $|z_1|=|z_2|$.
\begin{lem}\label{lem:2-d}
Let $\textbf{z}=(z_1,z_2,z_3,z_4)$ be a point in $S$ with $z_i\neq z_j$ for any $i,j$, $|z_1|>1$ and $|z_1|\geq |z_2|\geq|z_3|\geq|z_4|$. If $(\lambda_1,\lambda_2,\lambda_3)$ satisfies the condition $(A)$, then the condition $(B)$ holds.
\end{lem}
\begin{proof}

 By assumption, for any $\epsilon>0$, there exists a function $h\in L^2(\Gamma\backslash\mathcal{B}_4)$ with $\|h\|_{2,w}=1$ such that for any $i\in\{1,2,3\}$, the function 
$$\delta_i(v_{\ell,m,n}):=A_{w,i}h(v_{\ell,m,n})-\lambda_ih(v_{\ell,m,n})$$ satisfies $\|\delta_i\|_{2,w}<\epsilon$. Let $f$ be a simultaneous eigenfunction corresponding to $(\lambda_1,\lambda_2,\lambda_3)$ with $a:=f(v_{0,0,0})=h(v_{0,0,0}).$ Then we have the following lemma, of which proof will appear in Appendix \ref{ap:b}.
\begin{lem}\label{lem:6.77}
Let $\textbf{z}=(z_1,z_2,z_3,z_4)$ be a point in $S$ with $|z_1|\geq |z_2|\geq|z_3|\geq|z_4|$ and $|z_1|>1$. Let $C_1$, $C_2$, and $C_3$ be the sum of the coefficients of terms with the same absolute value as $|z_1|^\ell$, $|z_1z_2|^\ell$ and $|z_1z_2z_3|^\ell$ in $f(v_{\ell,0,0}),f(v_{\ell,\ell,0})$ and $f(v_{\ell,\ell,\ell})$, respectively. Let $m(z_1)$ be the multiplicity of $\sqrt{q}z_1$ in the equation \eqref{char poly}. For sufficiently small $\epsilon,c>0$, we have
\begin{equation}
\begin{split}
&\left|\frac{h(v_{\ell,0,0})}{q^{((3+c)\ell/2}}\right|\geq (|C_1|-C\epsilon \ell^{m(z_1)-1})\left|\frac{z_1}{q^{c/2}}\right|^\ell-o(|z_1|^\ell)\\
&\left|\frac{h(v_{\ell,\ell,0})}{q^{(2+c)\ell}}\right|\geq (|C_2|-C\epsilon)\left|\frac{z_1z_2}{q^c}\right|^\ell-o(|z_1z_2|^\ell)\\
&\left|\frac{h(v_{\ell,\ell,\ell})}{q^{(3\ell+c)/2}}\right|\geq (|C_3|-C\epsilon\ell^{m(z_1)-1})\left|\frac{z_1z_2z_3}{q^{c/2}}\right|^\ell-o(|z_1z_2z_3|^\ell)
\end{split}
\end{equation}
whenever $\|\delta_i\|_{2,w}<\epsilon$ and $A_i$ is nonzero.
\end{lem}
Let $C_{ijk}$ be the constant in Theorem \ref{thm:eigenftn}. If $z_i=qz_j$, then the constants $C_{i**}$, $C_{ki*}$ and $C_{kli}$ are zero when $k,l\neq j$. In this case, 
$\textbf{z}$ satisfies that the condition $(B)$ holds if and only if the absolute values of all remaining terms $z_i^\ell $, $z_i^\ell z_j^\ell $ and $z_i^\ell z_j^\ell z_k^\ell $ in $\frac{f(v_{\ell,0,0})}{q^{3\ell/2}}$, $\frac{f(v_{\ell,\ell,0})}{q^{2\ell}}$ and $\frac{f(v_{\ell,\ell,\ell})}{q^{3\ell/2}}$ are at most 1, respectively.

We confirm that either $C_i$ is nonzero for certain $i$, or that the condition $(B)$ holds in each case, except when $z_i=qz_4$ for some $i$.  Then Lemma \ref{lem:6.77} shows that if 
$C_i$ is nonzero for some $i$, either 
$\frac{h(v_{\ell,0,0})}{q^{3\ell/2}}$, $\frac{h(v_{\ell,\ell,0})}{q^{2\ell}}$ or $\frac{h(v_{\ell,\ell,\ell})}{q^{3\ell/2}}$ diverges as $\ell$ goes to infinity. 
This contradicts to the fact that $h$ is $L^2$. In the cases when $z_i=qz_4$ for some $i$, we will investigate whether either 
$\frac{h(v_{\ell,0,0})}{q^{3\ell/2}}$, $\frac{h(v_{\ell,\ell,0})}{q^{2\ell}}$ or $\frac{h(v_{\ell,\ell,\ell})}{q^{3\ell/2}}$ diverges, or if the condition $(B)$ holds.

\bigskip

\noindent\underline{Case 1. $|z_1|>|z_2|$:} Suppose that $z_1\neq qz_i$ for any $i$. By assumption, $|z_1|>1$ and the coefficient $C_1=A_{1,0,0}$ is nonzero, where $A_{1,0,0}$ is the constant in Appendix \ref{sec:appendix}.

Suppose that $z_1=qz_2$. 
\begin{itemize}
\item If $z_2=qz_3=q^2z_4$, then the condition $(B)$ holds since $C_{432}$ is the only nonzero constant and $|z_4|,|z_3z_4|,|z_2z_3z_4|<1$.
\item If $z_2=qz_3$ and $z_3\neq qz_4$, $C_{3}=aC_{321}$ is nonzero.
\item Suppose that $z_2=qz_4$. Then $C_{421}$ is nonzero. By \eqref{eq:b..4}, we have  
$$\left|\frac{h(v_{\ell,\ell,\ell})}{q^{3\ell/2}}\right|\geq \left| \sum_{i}^{\ell}\frac{d_{\ell-i}}{q^{\ell-i}}X_{124}\left(\frac{z_1z_2z_3}{z_1z_2z_4}\right)^i+aC_{421}\right||z_1z_2z_4|^\ell-o(|z_1z_2z_4|.^\ell).
$$
This shows that if $|z_2|>1$, then $\frac{h(v_{\ell,\ell,\ell})}{q^{3\ell/2}}$ diverges. Since $C_{342},C_{421},C_{423}$ and $C_{432}$ are nonzero and $z_2^3z_3=1$, the condition $(B)$ holds if $|z_2|=1$.
\item If $z_2\neq qz_3$ and $z_2\neq qz_4$, then $C_2=aC_{213}+aC_{214}$ is nonzero.
\end{itemize}

Suppose that $z_1=qz_3$ and $z_3\ne qz_4$. The coefficient $C_3=aC_{231}+aC_{312}+aC_{321}$ is nonzero. 

Suppose that $z_1=qz_3=q^2z_4$ or $z_1=qz_4$. Then the sum $C_{234}+C_{241}+C_{243}$ is nonzero. By \eqref{eq:b.3}, we have 
\begin{equation*}
\begin{split}
&\left|\frac{h(v_{\ell,0,0})}{q^{3\ell/2}}\right|\geq \left| \sum_{i}^{\ell}\frac{c_{\ell-i}}{q^{\ell-i}}X_1\left(\frac{z_1}{z_2}\right)^i+aC_{234}+aC_{241}+aC_{243}\right||z_2|^\ell-o(|z_2|^\ell).
\end{split}
\end{equation*}
This implies that $\frac{h(v_{\ell,0,0})}{q^{3\ell/2}}$ diverges when $|z_2|>1$. If $|z_2|=1$, then the condition $(B)$ holds since $C_{1ij}$, $C_{21i},$ $C_{231}$, $C_{31j}$ and $C_{321}$ are zero and the absolute values of the remaining terms are at most 1.
\bigskip

\noindent\underline{Case 2. $|z_1|=|z_2|$:} Suppose that $z_i\neq qz_j$ for any $i,j$, then $$C_2=aC_{123}+aC_{124}+aC_{213}+aC_{214}$$ is nonzero.

Suppose that $z_1=qz_3$.
\begin{itemize}
\item If $z_2=qz_4$, then $C_{ijk}$ are all zero except $C_{314}$, $C_{341}$, $C_{342}$, $C_{423}$, $C_{431}$ and $C_{432}$. The condtion $(B)$ holds since the absolute values of the remaining terms are at most 1.
\item If $z_3=qz_4$, $C_1=aC_{243}$ since $|z_1|=|z_2|$. 
\item If $z_2\neq qz_4$ and $z_3\neq qz_4$, $C_3=aC_{231}+aC_{312}+aC_{321}$ is nonzero.
\end{itemize}

Suppose that $z_1= qz_4$.
\begin{itemize}
\item If $z_2=qz_3$, then $C_{324}$, $C_{341}$, $C_{342}$, $C_{413}$, $C_{431}$ and $C_{432}$ are nonzero and the remaining coefficients are zero. The condtion $(B)$ holds since the absolute values of the remaining terms are at most 1.
\item If $z_2\neq qz_3$, $C_1=aC_{234}+aC_{241}+aC_{243}$ is nonzero.
\end{itemize}

Suppose that $z_1\neq qz_3$ and $z_1\neq qz_4$.
\begin{itemize}
\item If $z_2=qz_3$, then $C_1=aC_{132}+aC_{134}+aC_{143}$ is nonzero.
\item If $z_2=qz_4$, then $C_1=aC_{134}+aC_{142}+aC_{143}$ is nonzero.
\item If $z_2\ne qz_3$, $z_2\ne qz_4$ and $z_3=qz_4$, then $C_2=aC_{124}+aC_{214}$ is nonzero.
\end{itemize}
Thus we have Lemma \ref{lem:6.6}.
\end{proof}
 
By Lemma \ref{lem:6.6}, the following lemmata deal with the all possible cases in which the characteristic equation \eqref{char poly} has repeated roots and $|z_i|>1$ for some $i$. In the proof of lemmata, the polynomials $D_{ijk}$ and $F_{ijk}$ are from Proposition \ref{prop:a.2} and Proposition \ref{prop:a.3}, respectively.

\begin{lem}\label{lem:2-e}
Let $\textbf{z}=(z_1,z_2,z_3,z_4)$ be a point in $S$ with $z_1= z_2$. If $(\lambda_1,\lambda_2,\lambda_3)$ satisfies the condition $(A)$, then the condition $(B)$ holds. 
\end{lem}
\begin{proof}Let $h$ and $f$ be the functions in the proof of Lemma \ref{lem:2-d} and $a=h(v_{0,0,0})$. Let $C_1$, $C_2$ and $C_3$ be the coefficients in Lemma \ref{lem:6.77}. The proof is analogous to the proof of Lemma \ref{lem:2-d}. As an application of Lemma \ref{lem:6.77}, we will verify that either $C_i$ is nonzero for some $i$ or the condition $(B)$ holds in each case.
\bigskip

\noindent\underline{Case 1. $|z_1|>|z_3|$ and $|z_1|>|z_4|$:} Suppose that $z_1\neq qz_i$ for any $i$. Then the polynomial $C_2=aD_{113}(\ell,\ell,0)+aD_{114}(\ell,\ell,0)$ is nonzero. 

Suppose that $z_1=qz_3$ or $z_1=qz_4$. Without loss of generality, we may assume that $z_1=z_2=q\alpha e^{i\theta}$, $z_3=\alpha e^{i\theta}$ and $z_4=\frac{1}{q^2\alpha^3}e^{-3i\theta}$ for some $\theta$ and $\alpha>0$ with $q\alpha>1$. Since $\sigma_2(\textbf{z})$ is real, $\theta=\frac{\pi}{2}k$ for $k\in \mathbb{Z}$. Using this and the equation $\overline{\sigma_1(\textbf{z})}=\sigma_3(\textbf{z})$, we have $\alpha=\frac{1}{\sqrt{q}}$. 
 Then $z_1=z_2=qz_3=qz_4$. The polynomials $F_{331}$ and $F_{313}$ are nonzero. Then the condition $(B)$ holds since the absolute values of the remaining terms are at most 1. 
 \bigskip

 \noindent\underline{Case 2. $|z_3|>|z_1|$:} In this case, $z_3$ corresponds to $z_1$ in Lemma \ref{lem:6.77}. Suppose that $z_3\neq qz_i$ for any $i$. Then the polynomial $$C_1=aD_{311}(\ell,0,0)+aD_{314}(\ell,0,0)+aD_{341}(\ell,0,0)$$ is nonzero.
  
 Suppose that $z_3=qz_1$. Let $z_3=q\alpha e^{i\theta}$, $z_1=\alpha e^{i\theta}$ and $z_4=\frac{1}{q\alpha^3}e^{-3i\theta}$, where $\alpha>0$ with $\alpha>\frac{1}{\sqrt{q}}$. Since $\sigma_2(\textbf{z})$ is real, $\theta=\frac{\pi}{2}k$. This and the equation $\overline{\sigma_1(\textbf{z})}=\sigma_3(\textbf{z})$ imply that $\alpha=1$. The polynomials $D_{141}$, $D_{411}$ and $D_{413}$ is nonzero. 
 
 Suppose that $z_3=qz_4$. Let $z_3=q\alpha e^{i\theta}$, $z_4=\alpha e^{i\theta}$ and $z_1=\frac{1}{\sqrt{q} \alpha}e^{-i\theta}.$ The equation $\overline{\sigma_1(\textbf{z})}=\sigma_3(\textbf{z})$ holds only if $\alpha=\frac{1}{\sqrt{q}}.$ The nonzero polynomials are $D_{114}$, $D_{141}$, $D_{143}$, $D_{411}$, $D_{413}$ and $D_{431}.$ 
 
 In the case when $z_3=qz_1$ or $z_3=qz_4$, the absolute values of remaining terms are at most 1. This implies that the condition $(B)$ holds. Thus we have Lemma \ref{lem:2-e}.
\end{proof}
\subsection{From the condition $(B)$, passing through $(C)$ to the condition $(D)$}\label{sec:6.2}
Theorem \ref{thm:crucial} and Theorem \ref{thm:6.11} describe the complement of $\mathcal{S}^4$ in the simultaneous spectrum and allow us to show that if the simultaneous eigenfunction $f_{\textbf{z}}$ satisfies the condition $(B)$, then the condition $(C)$ holds. In the process of the proof of Theorem \ref{thm:crucial} and Theorem \ref{thm:6.11}, we establish that condition $(C)$ implies condition $(D)$.

\begin{thm}\label{thm:crucial}
Let $\textbf{z}=(z_1,z_2,z_3,z_4)$ be a point in $S$ satisfying $|z_1|\geq|z_2|\geq|z_3|\geq|z_4|$ and $z_i\neq z_j$ for any $i\neq j$. 
If the simultaneous eigenfunction $f_{\textbf{z}}$ satisfies the condition $(B)$, then the condition $(C)$ holds.
\end{thm}
\begin{proof} If $|z_1|=1$, we have $\textbf{z}=(e^{i\theta_1},e^{i\theta_2},e^{i\theta_3},e^{-i(\theta_1+\theta_2+\theta_3)})$. The rest part of the proof follows from two below lemmata.
\begin{lem} \label{lem:5.2}
Suppose that $z_i\neq z_j$ for any $i,j$, $|z_1|>1$ and $|z_1|>|z_2|\geq|z_3|\geq |z_4|$. Suppose that the simultaneous eigenfunction $f_{\textbf{z}}$ satisfies the condition $(B)$. 
The 4-tuple $(z_1,z_2,z_3,z_4)$, up to permutation, belongs to one of the following cases:
\begin{itemize}
\item $(q\sqrt{q}e^{\frac{k\pi}{2}i},\sqrt{q}e^{\frac{k\pi}{2}i},\frac{1}{\sqrt{q}}e^{\frac{k\pi}{2}i},\frac{1}{q\sqrt{q}}e^{\frac{k\pi}{2}i})$, $k\in \{0,1,2,3\}$,
\item $(qe^{i\theta},e^{i\theta},e^{-3i\theta},\frac{1}{q}e^{i\theta})$ with $e^{4i\theta}\neq 1$, or
\item $(\sqrt{q}e^{i\theta},e^{i\phi},e^{-i(\phi+2\theta)},\frac{1}{\sqrt{q}}e^{i\theta})$ with $e^{i\phi}\neq e^{-i\theta}$.
\end{itemize}
\end{lem}
The 4-tuple $\textbf{z}$ belonging to the cases above satisfies the condition $(D)$.
\begin{proof}As we mentioned in the proof of Lemma \ref{lem:2-d}, to satisfy the condition $(B)$, the coefficient of $z_i^\ell$, $z_i^\ell z_j^\ell$ and $z_i^\ell z_j^\ell z_k^\ell $ in $\frac{f(v_{\ell,0,0})}{q^{3\ell/2}}$, $\frac{f(v_{\ell,\ell,0})}{q^{2\ell}}$ and $\frac{f(v_{\ell,\ell,\ell})}{q^{3\ell/2}}$ need to be zero whenever $|z_i|,|z_iz_j|,|z_iz_jz_k|>1$ for some $i,j,k$.
Thus the coefficient $C_1$ of $z_1^\ell$ in $\frac{f(v_{\ell,0,0})}{q^{3\ell/2}}$ need to be zero. Hence, it follows that either $z_1=qz_2$, $z_1=qz_3$ or $z_1=qz_4$ holds.

\bigskip

\noindent \underline{Case 1. $z_1=qz_2$}: In this case, $C_{1ij}$, $C_{31i}$, $C_{341}$, $C_{41i}$ and $C_{431}$ are zero. Since $|z_1z_2z_3|=|z_4|^{-1}>1,$ 
the condition $(B)$ does not hold when $C_{213}+C_{231}+C_{321}\neq 0.$ The numerator of $C_{213}+C_{231}+C_{321}$ is of the form
$$(z_1-qz_4)(z_2-qz_1)(z_2-qz_4)(z_3-qz_4)(z_3-qz_2)(z_2-z_3)(q^2+q+1).$$
Under the assumption, $(z_1-qz_4)$, $(z_2-qz_1),$ $(z_3-qz_2)$ and $(z_2-z_3)$ are nonzero. Thus $C_{213}+C_{231}+C_{321}=0$ if and only if $z_2=qz_4$ or $z_3=qz_4$. 

If $z_2=qz_4$, then $C_{342}$, $C_{421}$, $C_{423}$ and $C_{432}$ are nonzero. In this case, if either $|z_3|>1$ or 
$|z_3|=|z_1z_2z_4|^{-1}<1$, the condition $(B)$ does not hold. 
Thus, we have $|z_3|=1$ and $$(z_1,z_2,z_3,z_4)=\left(qe^{i\theta},e^{i\theta},e^{-3i\theta},\frac{1}{q}e^{i\theta}\right).$$ Such a $4$-tuple $\textbf{z}$ satisfies $$|z_3^\ell z_4^m z_2^n|,|z_4^\ell z_2^m z_1^n|,|z_4^\ell z_2^m z_3^n|,|z_4^\ell z_3^m z_2^n|\leq 1$$
for every $\ell\geq m\geq n\ge 0$. Hence, 
$\textbf{z}$ satisfies the condition $(D)$.

Now let us consider the case $z_3=qz_4$. The coefficients $C_{214}$, $C_{241}$, $C_{243}$, $C_{421}$, $C_{423}$ and $C_{432}$ are remaining. Since $z_1z_2z_3z_4=1$, $|z_2z_4|=\frac{1}{q}\geq |z_3z_4|.$ 
This implies that $|z_1z_2|>1$ and $C_{214}$ needs to be zero. This is possible if and only if $z_2=qz_3$, so we have
$$(z_1,z_2,z_3,z_4)=\left(q\sqrt{q}e^{\frac{k\pi}{2}i},\sqrt{q}e^{\frac{k\pi}{2}i},\frac{1}{\sqrt{q}}e^{\frac{k\pi}{2}i},\frac{1}{q\sqrt{q}}e^{\frac{k\pi}{2}i}\right).$$
Since $f_{\textbf{z}}(v_{\ell,m,n})=C_{432}z_4^\ell z_3^m z_2^n$, it follows that $f_{\textbf{z}}\in L^2(\Gamma\backslash\mathcal{B}_4)$. The 4-tuple $\textbf{z}$ satisfies the condition $(D)$.

\noindent \underline{Case 2. $z_1=qz_3$}: In this case, $C_{1ij}$, $C_{21i}$, $C_{241}$, $C_{41i}$ and $C_{421}$ are zero. Similar to Case 1, $C_{231}+C_{312}+C_{321}$ needs to be zero. The numerator of $C_{231}+C_{312}+C_{321}$ is of the form
$$(z_1-qz_4)(z_2-qz_4)(z_3-qz_4)(z_3-qz_1)(z_2-qz_3)(z_2-z_3)(q^2+q+1).$$
Under the assumption, either $z_3=qz_4$ or $z_2=qz_4$ should hold. If $z_2=qz_4$, then it is of the form $(q^{\frac{1}{2}+\alpha}e^{i\theta},q^{\frac{1}{2}-\alpha}e^{-i\theta},q^{-\frac{1}{2}+\alpha}e^{i\theta},q^{-\frac{1}{2}-\alpha}e^{-i\theta})$ for some $\alpha\in (0,\frac{1}{2})$. We may exclude this case since $C_{314}\ne 0$ while $|z_1 z_3|>1$.

Hence, it follows that $z_1=qz_3=q^2z_4$. Following the method in the case when $z_1=qz_2=q^2z_4$, we have $\textbf{z}=(qe^{i\theta},e^{-3i\theta},e^{i\theta},\frac{1}{q}e^{i\theta}).$ The only nonzero coefficients are $C_{243},$ $C_{423},$ $C_{431}$ and $C_{432}$.  The 4-tuple $\textbf{z}$ satisfies the condition $(B)$ since for any $\ell\geq m\geq n$, $$|z_2^\ell z_4^m z_3^n|,|z_4^\ell z_2^m z_3^n|,|z_4^\ell z_3^m z_1^n|,|z_4^\ell z_3^m z_2^n|\leq 1.$$ 

\bigskip

\noindent \underline{Case 3. $z_1=qz_4$}: In this case, $C_{1ij},$ $C_{213}$, $C_{214}$, $C_{231}$, $C_{312}$, $C_{314}$ and $C_{321}$ are equal to zero and the remaining coefficients are nonzero. Since $C_{234}+C_{241}+C_{243}$ and $C_{241}+C_{412}+C_{421}$ 
are nonzero, if either $|z_2|>1$ or $|z_3|=|z_1z_2z_4|^{-1}<1$, the condition $(B)$ does not hold.  Hence we have $|z_2|=|z_3|=1$. This yields $$(z_1,z_2,z_3,z_4)=\left(\sqrt{q}e^{i\theta},e^{i\phi},e^{-i(2\theta+\phi)},\frac{1}{\sqrt{q}}e^{i\theta}\right).$$ For any $\ell\geq m\geq n\ge 0$, the inequality $|z_i^\ell z_j^mz_k ^n|\le 1 $ holds whenever $C_{ijk}\neq 0$. We complete the proof.
\end{proof}
\begin{lem}\label{lem:6.10}
Suppose that $z_i\neq z_j$ for any $i\neq j$, $|z_1|>1$ and $|z_1|=|z_2|\geq|z_3|\geq |z_4|$. Suppose that the simultaneous eigenfunction $f_{\textbf{z}}$ satisfies the condition $(B)$.
The 3-tuple $(z_1,z_2,z_3,z_4)$, up to permutation, is of the form
 $$(z_1,z_2,z_3,z_4)=(\sqrt{q}e^{i\theta},\pm\sqrt{q}e^{-i\theta},\frac{1}{\sqrt{q}}e^{i\theta},\pm\frac{1}{\sqrt{q}}e^{-i\theta})$$ with $e^{i\theta}\neq \pm e^{-i\theta}.$ The 4-tuple $\textbf{z}$ belonging to the cases above satisfies the condition $(D)$.
\end{lem}
\begin{proof}
By assumption, 
the coefficient $C_{123}+C_{124}+C_{213}+C_{214}$ of $z_1^\ell z_2^\ell$ in $\frac{f(v_{\ell,\ell,0})}{q^{2\ell}}$ needs to be zero. Since the numerator of $C_{123}+C_{124}+C_{213}+C_{214}$ is of the form
$$(q+1)^2(z_1-z_2)(z_1-qz_3)(z_1-qz_4)(z_2-qz_3)(z_2-qz_4)(z_3-z_4),$$
either $z_1=qz_3$, $z_1=qz_4$, $z_2=qz_3$ or $z_2=qz_4$ holds.
Let us consider the case that $z_1=qz_3$. Since $|z_2|>1$, $C_{231}+C_{234}+C_{243}$, of which the numerator is of the form
$$(z_2-qz_1)(z_2-qz_3)(z_2-qz_4)(z_3-qz_1)(z_4-qz_3)(z_4-z_3)(q^2+q+1),$$
needs to be zero. Thus we have  $z_2=qz_4$ and $$(z_1,z_2,z_3,z_4)=\left(\sqrt{q}e^{i\theta},\pm\sqrt{q}e^{-i\theta},\frac{1}{\sqrt{q}}e^{i\theta},\pm\frac{1}{\sqrt{q}}e^{-i\theta}\right)$$ with $e^{2i\theta}\neq \pm 1.$ The coefficients $C_{314}$, $C_{341}$, $C_{342}$, $C_{423}$, $C_{431}$ and $C_{432}$ are nonzero and the rest coefficients are zero. For any $\ell\geq m\geq n$, the terms satisfy that $$|z_3^\ell z_1^m z_4^n|, |z_3^\ell z_4^mz_1^n|, |z_3^\ell z_4^m z_2^n|,|z_4^\ell z_2^m z_3^n|,|z_4^\ell z_3^m z_1^n|,|z_4^\ell z_3^m z_2^n|\leq 1.$$
In the case when $z_1=qz_4$, we have $z_2=qz_3$ and
 $$(z_1,z_2,z_3,z_4)=\left(\sqrt{q}e^{i\theta},\pm\sqrt{q}e^{-i\theta},\pm\frac{1}{\sqrt{q}}e^{-i\theta},\frac{1}{\sqrt{q}}e^{i\theta}\right)$$ with $e^{2i\theta}\neq \pm 1.$ Thus 
 $\textbf{z}$ satisfies the condition $(D)$.
\end{proof}
By Lemma \ref{lem:6.6}, Lemma \ref{lem:5.2} and Lemma \ref{lem:6.10}, we have Theorem \ref{thm:crucial}.
\end{proof}
\begin{thm}\label{thm:6.11}
Let $\textbf{z}=(z_1,z_2,z_3,z_4)$ be a point in $S$ with $z_i=z_j$ for some $i,j$, $|z_1|>1$ and $|z_1|\geq |z_2|\geq|z_3|\geq|z_4|$. 
If the simultaneous eigenfunction $f_{\textbf{z}}$ satisfies the condition $(B)$, then the condition $(C)$ holds.
\end{thm}
\begin{proof} The proof follows from two below lemmata.
\begin{lem}\label{lem:6.12}Let $\textbf{z}=(z_1,z_2,z_3,z_4)$ be a point in $S$ with $|z_k|>1$ for some $k$ and $z_i\neq z_j$ except for $z_1=z_2$. Suppose that the simultaneous eigenfunction $f_{\textbf{z}}$ satisfies the condition $(B)$.
The 4-tuple $(z_1,z_2,z_3,z_4)$ is of the from
\begin{itemize}
\item $\left(e^{\frac{\pi k}{2}i},e^{\frac{\pi k}{2}i},qe^{\frac{\pi k}{2}i},\frac{1}{q}e^{\frac{\pi k}{2}i}\right)$ for $k\in \mathbb{Z}$,
\item $\left(e^{-i\theta},e^{-i\theta},\sqrt{q}e^{i\theta},\frac{1}{\sqrt{q}}e^{i\theta}\right).$
\end{itemize}
The 4-tuple $\textbf{z}$ belonging to the cases above satisfies the condition $(D)$.
\end{lem}
\begin{proof}The proof is analogous to Lemma \ref{lem:5.2} and Lemma \ref{lem:6.10}. The eigenfunction $f$ of a simultaneous eigenvalue $(\lambda_1,\lambda_2,\lambda_3)$ is given in Proposition \ref{prop:a.2}. We consider two cases that $|z_1|$ is the maximal modulus or not.
\bigskip

\noindent \underline{Case 1. $|z_1|>|z_3|$ and $|z_1|>|z_4|$}: The coefficient $C_2$ of $z_1^{2\ell}$ in $\frac{f(v_{\ell,\ell,0})}{q^{2\ell}}$ needs to be zero. The numerator of $C_2=D_{113}(\ell,\ell,0)+D_{114}(\ell,\ell,0)$ is of the from
$$-(q+1)^2(z_1-qz_3)^2(z_1-qz_4)^2(z_3-z_4).$$
Thus $C_2=0$ only if $z_1=qz_3$ or $z_1=qz_4$. As in the proof of Lemma \ref{lem:2-e}, we have $z_1=z_2=qz_3=qz_4$. It does not fulfill the assumption that $z_3\neq z_4$. 
\bigskip

\noindent \underline{Case 2. $|z_3|>|z_1|$}: The coefficient $C_1$ of $z_3^\ell$ in $\frac{f(v_{\ell,0,0})}{q^{3\ell/2}}$ needs to be zero. 
The numerator of $C_1=D_{311}(\ell,0,0)+D_{314}(\ell,0,0)+D_{341}(\ell,0,0)$ is equal to 
$$-(z_3-qz_1)^2(z_3-qz_4)(z_1-z_4)^2(q^2+q+1).$$
It follows that either $z_3=qz_1$ or $z_3=qz_4$ holds.

Suppose that $z_3=qz_1$. In the proof of Lemma \ref{lem:2-e}, we have $z_3=qz_1=q^2z_4$. Thus the tuple $\textbf{z}$ is of the form $$\textbf{z}=\left(e^{\frac{\pi k}{2}i},e^{\frac{\pi k}{2}i},qe^{\frac{\pi k}{2}i},\frac{1}{q}e^{\frac{\pi k}{2}i}\right)\text{ for }k\in \mathbb{Z}.$$
The polynomials $D_{141}$, $D_{411}$ and $D_{413}$ is nonzero and the remaining polynomials are zero. Since for any $\ell\geq m\geq n\geq 0$,
$$|z_1^{\ell+n}z_4^m|,|z_4^\ell z_1^{m+n}|,|z_4^\ell z_1^m z_3^n|\leq 1,$$
the 4-tuple $\textbf{z}$ satisfies the condition $(D)$. 

Suppose that $z_3=qz_4$. In the proof of Lemma \ref{lem:2-e}, we have $|z_1|=1$. This shows that $\textbf{z}$ is of the form
$$\textbf{z}=\left(e^{-i\theta},e^{-i\theta},\sqrt{q}e^{i\theta},\frac{1}{\sqrt{q}}e^{i\theta}\right).$$
For any $\ell\geq m\geq n\geq0$,$$|z_1^{\ell+m}z_4^n|,|z_1^{\ell+n}z_4^m|,|z_1^\ell z_4^mz_3^n|,|z_4^\ell z_1^{m+n}|,|z_4^\ell,z_1^mz_3^n|,|z_4^\ell z_3^m z_1^n|\leq 1.$$ 
Thus the condition $(D)$ holds. 
This completes proof.
\end{proof} 
\begin{lem}\label{lem:6.13}Let $\textbf{z}=(z_1,z_2,z_3,z_4)$ be a point in $S$ with $z_1=z_2$, $z_3=z_4$ and $|z_1|>|z_3|$.
 Suppose that the simultaneous eigenfunction $f_{\textbf{z}}$ satisfies the condition $(B)$.
The 4-tuple $(z_1,z_2,z_3,z_4)$ is of the form
$$\left(\sqrt{q}e^{\frac{\pi}{2}ki},\sqrt{q}e^{\frac{\pi}{2}ki},\frac{1}{\sqrt{q}}e^{\frac{\pi}{2}ki},\frac{1}{\sqrt{q}}e^{\frac{\pi}{2}ki}\right)\text{ for } k\in \mathbb{Z}.$$
and satisfies the condition $(D)$.
\end{lem}
\begin{proof}Similarly, if $F_{113}$ is nonzero, the condition $(B)$ does not hold. The polynomial $F_{113}$ is zero only if $z_1=qz_3$. In this case, $F_{313}$ and $F_{331}$ are the nonzero polynomials and we have
$|z_3^{\ell+n}z_1^m|,|z_3^{\ell+m}z_1^n|\leq 1$ for any $\ell \geq m\geq n\geq 0$. This implies that $\textbf{z}$ satsifies condition $(D)$. 
\end{proof}
Using Lemma \ref{lem:6.12} and Lemma \ref{lem:6.13}, we complete the proof.
\end{proof}
\subsection{From the condition $(D)$ to the condition $(A)$}\label{sec:6.3}
Proposition \ref{prop:spectest}, Corollary \ref{coro:6.4} and Corollary \ref{coro:6.5} enable to show that if the simultaneous eigenfunction $f_\textbf{z}$ indexed by 4-tuple $\textbf{z}$ satisfies the condition $(D)$, the condition $(A)$ holds. 
\begin{prop}\label{prop:spectest}
Let $\textbf{z}=(z_1,z_2,z_3,z_4)$ be a point in $S$ for which $z_i\neq z_j$ for every $i\ne j$. 
If the eigenfunction $f_\textbf{z}$ satisfies the condition $(D)$, the condition $(A)$ holds.
\end{prop}
\begin{proof}Suppose that $|z_i^\ell z_j^mz_k^n|< 1$ for any $i,j,k$ with $C_{ijk}\neq 0$. Then it is not difficult to see that $\|f_\textbf{z}\|_{2,w}<\infty$. 

Now consider the case when $|z_i^\ell z_j^mz_k^n|=1$ for some $i,j,k$ with $C_{ijk}\neq 0$. In this case, $f_\textbf{z}$ is not contained in $L^2_w(\Gamma\backslash\mathcal{B}_4)$. Define a function $f_{\textbf{z}}^\epsilon$ by $$f_{\textbf{z}}^\epsilon (v_{\ell,m,n})=(1-\epsilon)^\ell f_{\textbf{z}}(v_{\ell,m,n}).$$
By assumption, $|f_{\textbf{z}}(v_{\ell,m,n})|^2w(v_{\ell,m,n})\leq C:=\left(\sum |C_{ijk}|\right)^2$ and $ f_{\textbf{z}}^\epsilon$ is $L^2$. We claim that for any $\epsilon\in(0,1/2)$, the following inequality holds:
$$\|A_{w,i}f_{\textbf{z}}^\epsilon-\lambda_i f_{\textbf{z}}^\epsilon\|_{2,w}\leq \epsilon(1-\epsilon)^{-1} (q^4+q^3+2q^2+q+1)\|f_{\textbf{z}}^\epsilon\|_{2,w}.$$
From the above inequality, we have $$\displaystyle\lim_{\epsilon \rightarrow0+}\frac{\|A_{w,i}f_{\textbf{z}}^\epsilon-\lambda_i f_{\textbf{z}}^\epsilon\|_{2,w}}{\|f_{\textbf{z}}^\epsilon\|_{2,w}}=0.$$
 
Let us prove the claim. Since $f_{\textbf{z}}$ is the eigenfunction of the simultaneous eigenvalue $\lambda_i$ associated to $A_{w,i}$, we have $\lambda_if_\textbf{z}^\epsilon(v_{\ell,m,n})=(1-\epsilon)^\ell A_{w,i}f_\textbf{z}(v_{\ell,m,n}).$ By the definition of $A_{w,i}$, the following equations hold:
 \begin{equation}
 \begin{split}\nonumber
 A_{w,1}f_{\textbf{z}}^\epsilon(v_{0,0,0})-\lambda_{1}f_{\textbf{z}}^\epsilon(v_{0,0,0})&=-\epsilon(q^3+q^2+q+1)f_\textbf{z}(v_{1,0,0})\\
  A_{w,1}f_{\textbf{z}}^\epsilon(v_{\ell,0,0})-\lambda_{1}f_{\textbf{z}}^\epsilon(v_{\ell,0,0})&=-\epsilon(1-\epsilon)^\ell f_\textbf{z}(v_{\ell+1,0,0})\\
  A_{w,1}f_{\textbf{z}}^\epsilon(v_{\ell,\ell,0})-\lambda_{1}f_{\textbf{z}}^\epsilon(v_{\ell,\ell,0})&=-\epsilon(1-\epsilon)^\ell(q+1)f_\textbf{z}(v_{\ell+1,\ell,0})\\
   A_{w,1}f_{\textbf{z}}^\epsilon(v_{\ell,\ell,\ell})-\lambda_{1}f_{\textbf{z}}^\epsilon(v_{\ell,\ell,\ell})&=\epsilon(1-\epsilon)^{\ell-1}q^3f_\textbf{z}(v_{\ell-1,\ell-1,\ell-1})\\&\qquad\qquad-\epsilon(1-\epsilon)^\ell(q^2+q+1)f_\textbf{z}(v_{\ell+1,\ell,\ell})\\
   A_{w,1}f_{\textbf{z}}^\epsilon(v_{\ell,m,0})-\lambda_{1}f_{\textbf{z}}^\epsilon(v_{\ell,m,0})&=-\epsilon(1-\epsilon)^\ell f_\textbf{z}(v_{\ell+1,m,0})\\
      A_{w,1}f_{\textbf{z}}^\epsilon(v_{\ell,m,m})-\lambda_{1}f_{\textbf{z}}^\epsilon(v_{\ell,m,m})&=\epsilon(1-\epsilon)^{\ell-1}q^3 f_\textbf{z}(v_{\ell+1,m,0})-\epsilon(1-\epsilon)^\ell f_\textbf{z}(v_{\ell+1,m,m})\\  
         A_{w,1}f_{\textbf{z}}^\epsilon(v_{\ell,\ell,m})-\lambda_{1}f_{\textbf{z}}^\epsilon(v_{\ell,\ell,m})&=\epsilon(1-\epsilon)^{\ell+1} q^3 f_\textbf{z}(v_{\ell-1,\ell-1,m-1})-\epsilon(1-\epsilon)^\ell f_\textbf{z}(v_{\ell+1,\ell,m})\\
   A_{w,1}f_{\textbf{z}}^\epsilon(v_{\ell,m,n})-\lambda_{1}f_{\textbf{z}}^\epsilon(v_{\ell,m,n})&=\epsilon(1-\epsilon)^{\ell-1} q^3f_\textbf{z}(v_{\ell-1,m-1,n-1})-\epsilon(1-\epsilon)^\ell f_{\textbf{z}}(v_{\ell+1,m,n}).\\
 \end{split}
 \end{equation}
 This proves the claim when $i=1$. Similar argument holds for $i=2$ and $3$.
\end{proof}

\begin{coro}\label{coro:6.4}Let $\textbf{z}=(z_1,z_2,z_3,z_4)$ be a point in $S$ with $z_i\neq z_j$ except for $z_1=z_2$. 
If the eigenfunction $f_\textbf{z}$ satisfies the condition $(D)$, the condition $(A)$ holds.
\end{coro}
\begin{proof}By assumption and Proposition \ref{prop:a.2}, $|f(v_{\ell,m,n})|^2w(v_{\ell,m,n})\leq (A\ell+B)^2$ for some $A,B>0$. Following the proof of Proposition \ref{prop:spectest}, we have Corollary \ref{coro:6.4}.
\end{proof}
\begin{coro}\label{coro:6.5}Let $\textbf{z}=(z_1,z_2,z_3,z_4)$ be a point in $S$ with $z_1= z_2$, $z_3=z_4$ and $z_1\neq z_3$. 
If the eigenfunction $f_\textbf{z}$ satisfies the condition $(D)$, the condition $(A)$ holds.
\end{coro}
\begin{proof}By assumption and Proposition \ref{prop:a.3}, $|f(v_{\ell,m,n})|^2w(v_{\ell,m,n})\leq (A\ell^2+B\ell+C)^2$ for some $A,B,C>0$. Following the proof of Proposition \ref{prop:spectest}, we have Corollary \ref{coro:6.5}.
\end{proof}
From the above crucial observation, we may describe all elements in the automorphic simultaneous spectrum of $A_{w,i}$.


\appendix

\section{Simultaneous eigenfunctions via recurrence relation}\label{sec:appendix}

In this appendix, we give the proof of Theorem~\ref{thm:eigenftn} as well as some remarks on the multiple root cases.


\begin{thm}\label{thm:eigenftn(a)}
 Let $(\lambda_1,\lambda_2,\lambda_3)$ be an element of $\mathbb{C}^3$ indexed by the equation~(\ref{eq:para}) with respect to $(z_1,z_2,z_3,z_4)$. If $f$ is the corresponding simultaneous eigenfunction satisfying $A_{w,i}f=\lambda_if$, then  
$$f(v_{\ell,m,n})=\sum_{\tau\in S_4}C_{\tau(1)\tau(2)\tau(3)}\sqrt{q}^{3\ell+m-n}z_{\tau(1)}^\ell z_{\tau(2)}^m z_{\tau(3)}^n.$$
\end{thm}

\begin{proof}

 By Lemma \ref{lem:4.6}, the eigenfunction $f=f_{\textbf{z}}$ of simultaneous eigenvalue $\lambda_i$ of $A_{w,i}$ satisfies that
\begin{equation}\nonumber
\begin{split}
&\lambda_1f(v_{\ell+1,0,0})=(q^3+q^2+q)f(v_{\ell+1,1,0})+f(v_{\ell+2,0,0})\\
&\lambda_2f(v_{\ell,0,0})=(q^4+q^3+q^2)f(v_{\ell,1,1})+(q^2+q+1)f(v_{\ell+1,1,0})\\
&\lambda_3f(v_{\ell-1,0,0})=q^3f(v_{\ell-2,0,0})+(q^2+q+1)f(v_{\ell,1,1}).
\end{split}
\end{equation}
Let $a_{\ell,m,n}=f(v_{\ell,m,n})/q^{\ell+m+n}.$ Using above equation, we have
\begin{equation}\label{eq:5.1}
\begin{split}
&\sqrt{q}\sigma_1(\textbf{z})a_{\ell+1,0,0}=(q^3+q^2+q)a_{\ell+1,1,0}+a_{\ell+2,0,0}\\
&{q}\sigma_2(\textbf{z})a_{\ell,0,0}=(q^5+q^4+q^3)a_{\ell,1,1}+(q^3+q^2+q)a_{\ell+1,1,0}\\
& q^{3/2}\sigma_3(\textbf{z})a_{\ell-1,0,0}=q^2a_{\ell-2,0,0}+(q^5+q^4+q^3)a_{\ell,1,1}.
\end{split}
\end{equation}
This implies that 
\begin{equation}\label{eq:5.111}
a_{\ell+2,0,0}-{q}^{1/2}\sigma_1(\textbf{z})a_{\ell+1,0,0}+q\sigma_2(\textbf{z})a_{\ell,0,0}-q^{3/2}\sigma_3(\textbf{z})a_{\ell-1,0,0}+q^2\sigma_4(\textbf{z})a_{\ell-2,0,0}=0.
\end{equation}

The sequence $a_{\ell,0,0}$ is of the form 
$$a_{\ell,0,0}=\sum_{i}A_{i,0,0}q^{\ell/2}z_i^\ell.$$
By Lemma \ref{lem:4.5} and Lemma \ref{lem:4.6}, 
\begin{equation}
\begin{split}\nonumber
a_{1,0,0}=&\frac{\lambda_1}{q(q^3+q^2+q+1)}=\frac{\sqrt{q}\sigma_1(\textbf{z})}{q^3+q^2+q+1}\\
a_{2,0,0}=&\frac{\lambda_1}{q}a_{1,0,0}-(q^3+q^2+q)a_{1,1,0}=\frac{q\sigma_1(\textbf{z})^2}{q^3+q^2+q+1}-\frac{q\sigma_2(\textbf{z})}{q^2+1}\\
a_{3,0,0}=&\frac{\lambda_1}{q}a_{2,0,0}-(q^3+q^2+q)a_{2,1,0}=\frac{\lambda_1}{q}a_{2,0,0}-\frac{\lambda_2}{q}{a_{1,0,0}}+(q^5+q^4+q^3)a_{1,1,1}\\
=&\frac{q^{3/2}\sigma_1(\textbf{z})^3}{q^3+q^2+q+1}-\frac{q^{3/2}\sigma_1(\textbf{z})\sigma_2(\textbf{z})}{q^2+1}-\frac{q^{3/2}\sigma_1(\textbf{z})\sigma_2(\textbf{z})}{q^3+q^2+q+1}+\frac{q^{3/2}(q^2+q+1)\sigma_3(\textbf{z})}{q^3+q^2+q+1}.
\end{split}
\end{equation}
Since 
\begin{equation}
\begin{split}\nonumber
&A_{1,0,0}+A_{2,0,0}+A_{3,0,0}+A_{4,0,0}=1\\
&z_1A_{1,0,0}+z_2A_{2,0,0}+z_3A_{3,0,0}+z_4A_{4,0,0}=\frac{a_{1,0,0}}{\sqrt{q}}\\
&z_1^2A_{1,0,0}+z_2^2A_{2,0,0}+z_3^2A_{3,0,0}+z_4^2A_{4,0,0}=\frac{a_{2,0,0}}{q}\\
&z_1^3A_{1,0,0}+z_2^3A_{2,0,0}+z_3^3A_{3,0,0}+z_4^3A_{4,0,0}=\frac{a_{3,0,0}}{q\sqrt{q}},
\end{split}
\end{equation}
we have $$A_{i,0,0}=\frac{(z_i-qz_k)(z_i-qz_j)(z_i-qz_l)}{(z_i-z_k)(z_i-z_j)(z_i-z_l)(q^3+q^2+q+1)},$$
where $i,j,k,l$ are distinct numbers in $\{1,2,3,4\}.$
 From the first and the third equation of \eqref{eq:4,8}, we have
 \begin{equation}\nonumber
 \begin{split}
 &\frac{\lambda_1}{q}a_{\ell+1,m+1,0}=(q^3+q^2)a_{\ell+1,m+1,1}+qa_{\ell+1,m+2,0}+a_{\ell+2,m+1,0}\\
 &\frac{\lambda_3}{q}a_{\ell,m,0}=qa_{\ell-1,m,}+a_{\ell,m-1,0}+(q^3+q^2)a_{\ell+1,m+1,1}.
  \end{split}
 \end{equation}
This implies that  
 \begin{equation}\label{eq:5.2}
 \begin{split}
 a_{\ell+2,m+1,0}+qa_{\ell+1,m+2,0}-\frac{\lambda_1}{q}a_{\ell+1,m+1,0}+\frac{\lambda_3}{q}a_{\ell,m,0}-qa_{\ell-1,m,0}-a_{\ell,m-1,0}=0.
  \end{split}
 \end{equation}
  Denote
 \begin{equation}\nonumber
 a_{\ell,m,0}=A_{1,m,0}(\sqrt{q}z_1)^\ell+A_{2,m,0}(\sqrt{q}z_2)^\ell+A_{3,m,0}(\sqrt{q}z_3)^\ell+A_{4,m,0}(\sqrt{q}z_4)^\ell.
 \end{equation}
 By \eqref{eq:5.2}, $A_{i,m,0}$ satisfies that 
 $$qA_{i,m+1,0}z_i^{2}+q^{3/2}A_{i,m+2,0}z_i-\frac{\lambda_1}{\sqrt{q}}A_{i,m+1,0}z_i+\frac{\lambda_3}{q}A_{i,m,0}-\sqrt{q}A_{i,m,0}z_i^{-1}-A_{i,m-1,0}=0.$$
 For any distinct numbers $i,j,k,l\in\{1,2,3,4\},$ we have
  $$q^{3/2}A_{i,m+2,0}-q(z_j+z_k+z_l)A_{i,m+1,0}+\sqrt{q}(z_jz_k+z_kz_l+z_lz_j)A_{i,m,0}-A_{i,m-1,0}z_i^{-1}=0.$$
Thus the sequence $A_{i,m,0}$ is of the form
$$A_{i,m,0}=B_{i,j,0}\biggl(\frac{z_j}{\sqrt{q}}\biggr)^m+B_{i,k,0}\biggl(\frac{z_k}{\sqrt{q}}\biggr)^m+B_{i,l,0}\biggl(\frac{z_l}{\sqrt{q}}\biggr)^m.$$
Applying the first equation of \eqref{eq:5.1}, we have $A_{i,1,0}:=\frac{z_j+z_k+z_l}{\sqrt{q}(q^2+q+1)}A_{i,0,0}.$ 
By the third equation of \eqref{eq:4.5} and the first equation of \eqref{eq:4.8},
we have 
\begin{equation}
\begin{split}\nonumber
&\frac{\lambda_3}{q^2}a_{1,0,0}=a_{0,0,0}+(q^3+q^2+q)a_{2,1,1}\\
&\frac{\lambda_1}{q}a_{2,1,0}=(q^3+q^2)a_{2,1,1}+qa_{2,2,0}+a_{3,1,0}.
\end{split}
\end{equation}
This implies that 
$$a_{2,2,0}=\frac{\lambda_1}{q^2}a_{2,1,0}-\frac{\lambda_3(q+1)}{q^4+q^3+q^2}a_{1,0,0}+\frac{q+1}{q^2+q+1}a_{0,0,0}-\frac{1}{q}a_{3,1,0}$$
For any distinct numbers $i,j,k,l\in \{1,2,3,4\}$, the above equation shows that 
\begin{equation}
\begin{split}\nonumber
qz_i^2A_{i,2,0}=\frac{\lambda_1}{q}z_i^2A_{i,1,0}-\frac{\lambda_3(q+1)}{q\sqrt{q}(q^2+q+1)}z_iA_{i,0,0}+\frac{q+1}{q^2+q+1}A_{i,0,0}-\sqrt{q}z_i^3A_{i,1,0}.
\end{split}
\end{equation}
For any distinct numbers $i,j,k,l\in\{1,2,3,4\},$
\begin{equation}
\begin{split}\nonumber
qA_{i,2,0}=&\frac{\lambda_1}{q}A_{i,1,0}-\frac{\lambda_3(q+1)}{q\sqrt{q}(q^2+q+1)}z_i^{-1}A_{i,0,0}+\frac{q+1}{q^2+q+1}z_i^{-2}A_{i,0,0}-\sqrt{q}z_iA_{i,1,0}\\
=&\biggl\{\frac{(z_i+z_j+z_k+z_l)(z_i+z_j+z_l)}{q^2+q+1}-\frac{(q+1)(z_jz_k+z_jz_l+z_kz_l+z_i^{-2})}{q^2+q+1}\\
&+\frac{q+1}{q^2+q+1}z_i^{-2}-\frac{z_i(z_j+z_k+z_l)}{q^2+q+1}\biggr\}A_{i,0,0}\\
=&\frac{(z_j+z_k+z_l)^2-(q+1)(z_jz_k+z_jz_l+z_kz_l)}{q^2+q+1}A_{i,0,0}.
\end{split}
\end{equation}
The second equation above follows from $A_{i,1,0}:=\frac{z_j+z_k+z_l}{\sqrt{q}(q^2+q+1)}A_{i,0,0}$ and $z_1z_2z_3z_4=1$.
Since 
\begin{equation}
\begin{split}
&B_{i,j,0}+B_{i,k,0}+B_{i,l,0}=A_{i,0,0}\\
&z_jB_{i,j,0}+z_kB_{i,k,0}+z_lB_{i,l,0}=\sqrt{q}A_{i,1,0}\\
&z_j^2B_{i,j,0}+z_k^2B_{i,k,0}+z_l^2B_{i,l,0}=qA_{i,2,0},
\end{split}
\end{equation}
we have that 
$$B_{i,j,0}=\frac{(z_j-qz_k)(z_j-qz_l)}{(z_j-z_k)(z_j-z_l)(q^2+q+1)}A_{i,0,0}.$$
Denote $A_{i,m,n}=\sum_{i\neq j} B_{i,j,n}\sqrt{q}^{-m}z_j^m.$ 
It follows from the first equation of Lemma~\ref{lem:4.11} that
$$\lambda_1a_{\ell,m,n}=a_{\ell-1,m-1,n-1}+q^3a_{\ell,m,n+1}+q^2a_{\ell,m+1,n}+qa_{\ell+1,m,n}.$$
The above equation shows that for any distinct numbers $i,j,k,l\in \{1,2,3,4\}$, 
$$\lambda_1B_{i,j,n}=B_{i,j,n-1}z_i^{-1}z_j^{-1}+q^3B_{i,j,n+1}+q^{3/2}B_{i,j,n}z_j+q^{3/2}B_{i,j,n}z_i.$$
Since $z_1z_2z_3z_4=1$, 
$$q^3B_{i,j,n+1}-q^{3/2}(z_k+z_l)B_{i,j,n}+z_kz_lB_{i,j,n-1}=0.$$
Thus the sequence $B_{i,j,n}$ is of the form
$$B_{i,j,n}=C_{ijk}\biggl(\frac{z_k}{q\sqrt{q}}\biggr)^n+C_{i,j,l}\biggl(\frac{z_l}{q\sqrt{q}}\biggr)^n.$$
By the second equation of \eqref{eq:4.8}, we have
$$\frac{\lambda_2}{q^2}a_{\ell,m,0}=a_{\ell-1,m-1,0}+(q^3+q^2)a_{\ell,m+1,1}+(q^2+q)a_{\ell+1,m,1}+a_{\ell+1,m+1,0}.$$
Using this, we have for any distinct $i,j,k,l\in\{1,2,3,4\}$
$$\sigma_2(\textbf{z})B_{i,j,0}=z_kz_lB_{i,j,0}+q^{3/2}(q^2+q)(z_i+z_j)B_{i,j,1}+z_iz_jB_{i,j,0}.$$
Thus we have 
$$B_{i,j,1}=\frac{z_k+z_l}{q^{3/2}(q+1)}B_{i,j,0}.$$
Since 
$C_{ijk}+C_{i,j,l}=B_{i,j,0}$ and $z_kC_{ijk}+z_lC_{i,j,l}=q^{3/2}B_{i,j,1},$ we have $$C_{ijk}=\frac{z_k-qz_l}{(z_k-z_l)(q+1)}B_{i,j,0}$$
and
$$a_{\ell,m,n}=\sum_{\tau\in z_4}C_{\tau(1),\tau(2),\tau(3)}\sqrt{q}^{\ell-m-3n}z_{\tau(1)}^\ell z_{\tau(2)}^mz_{\tau(3)}^n.$$
This completes the proof of the Theorem.
\end{proof}

Now let us find eigenfunctions when the characteristic equation \eqref{eq:5.111} has multiple root.
\begin{prop}\label{prop:a.2}
Suppose that $z_1=z_2$ and $z_1$, $z_3$ and $z_4$ are distinct. The eigenfunction of simultaneous eigenvalues $(\lambda_1,\lambda_2,\lambda_3)$ indexed by $(z_1,z_2,z_3,z_4)$ is 
\begin{equation}\label{eq:5.5}
\begin{split}
f(v_{\ell,m,n})=\sqrt{q}^{3\ell+m-n}\sum_{\substack{(i,j)\in\{3,4\}^2\\i\neq j}}&\left(D_{11i}z_1^{\ell+m}z_i^n
+D_{1i1}z_1^{\ell+n}z_i^m+D_{1ij}z_1^\ell z_i^mz_j^n+D_{i11}z_i^\ell z_1^{m+n}\right.\\
&\left.+D_{i1j}z_i^\ell z_1^m z_j^n+D_{ij1}z_i^\ell z_j^mz_1^n\right),\\
\end{split}
\end{equation}
where $D_{ijk}(\ell,m,n)$ are polynomials defined by
\begin{equation}
\begin{split}\nonumber
D_{11i}=&\frac{(-1)^i(z_1-qz_3)^2(z_1-qz_4)^2(z_i-qz_j)\{q+1+(\ell-m)(1-q)\}}{(z_1-z_3)^2(z_1-z_4)^2(z_3-z_4)(q+1)(q^2+q+1)(q^3+q^2+q+1)}\\
D_{1i1}=&\frac{(-1)^j(z_1-qz_j)^2(z_i-qz_j)}{(z_1-z_3)^2(z_1-z_4)^2(z_3-z_4)(q+1)(q^2+q+1)(q^3+q^2+q+1)}\\
&\times [\{q+1+(\ell-n)(1-q)\}(z_1-qz_i)(z_i-qz_1)+(1-q)(1-q^2)z_1z_i]\\
D_{1ij}=&\frac{(-1)^i(z_i-qz_j)}{(z_1-z_3)^2(z_1-z_4)^2(z_3-z_4)(q+1)(q^2+q+1)(q^3+q^2+q+1)}\\& \times [\{q+1+\ell(1-q)\}(z_1-qz_3)(z_3-qz_1)(z_1-qz_4)(z_4-qz_1)\\
&+(1-q)(1-q^2)\{(z_1-qz_3)(z_3-qz_1)z_1z_4+(z_1-qz_4)(z_4-qz_1)z_1z_3\}]
\end{split}
\end{equation}
\begin{equation}
\begin{split}\nonumber
D_{i11}=&\frac{(-1)^i(z_i-qz_1)^2(z_1-qz_j)^2(z_i-qz_j)[q+1+(m-n)(1-q)]}{(z_1-z_3)^2(z_1-z_4)^2(z_3-z_4)(q+1)(q^2+q+1)(q^3+q^2+q+1)}\\
D_{i1j}=&\frac{(-1)^j(z_i-qz_1)^2(z_i-qz_j)}{(z_1-z_3)^2(z_1-z_4)^2(z_3-z_4)(q+1)(q^2+q+1)(q^3+q^2+q+1)}\\
&\times[\{1+q+m(1-q)\}(z_1-qz_j)(z_j-qz_1)+(1-q)(1-q^2)z_1z_j]\\
D_{ij1}=&\frac{(-1)^i(z_3-qz_1)^2(z_4-qz_1)^2(z_i-qz_j)[q+1+n(1-q)]}{(z_1-z_3)^2(z_1-z_4)^2(z_3-z_4)(q+1)(q^2+q+1)(q^3+q^2+q+1)}.
\end{split}
\end{equation}
\end{prop}
\begin{proof}
Since for any $\ell\geq m\geq n$, 
\begin{equation}
\begin{split}\nonumber
&\lim_{z_2\rightarrow z_1}C_{12i}z_1^\ell z_2^mz_i^n+C_{21i}z_2^\ell z_1^m z_i^n=D_{11i}z_1^{\ell+m}z_i^n\\
&\lim_{z_2\rightarrow z_1}C_{1i2}z_1^\ell z_i^mz_2^n+C_{2i1}z_2^\ell z_i^m z_1^n=D_{1i1}z_1^{\ell+n}z_i^m\\
&\lim_{z_2\rightarrow z_1}C_{1ij}z_1^\ell z_i^mz_j^n+C_{2ij}z_2^\ell z_i^m z_j^n=D_{1ij}z_1^\ell z_i^mz_j^n\\
&\lim_{z_2\rightarrow z_1} C_{i12}z_i^\ell z_1^mz_2^n+C_{i21}z_i^\ell z_2^m z_1^n=D_{i11}z_i^\ell z_1^{m+n}\\
&\lim_{z_2\rightarrow z_1}C_{i1j}z_i^\ell z_1^m z_j^n+C_{i2j}z_i^\ell z_2^m z_j^n=D_{i1j}z_i^\ell z_1^m z_j^n\\
&\lim_{z_2\rightarrow z_1} C_{ij1}z_i^\ell z_j^mz_1^n+C_{ij2}z_i^\ell z_j^m z_2^n=D_{ij1}z_i^\ell z_j^mz_1^n\\
\end{split}
\end{equation}
by L'H\^opital's law,
we have \eqref{eq:5.5}.
\end{proof}
\begin{prop}\label{prop:a.3}Suppose that $z_1=z_2$,$z_3=z_4$ and $z_1$ and $z_3$ are distinct. The eigenfunction of simultaneous eigenvalues $(\lambda_1,\lambda_2,\lambda_3)$ indexed by $(z_1,z_2,z_3,z_4)$ is 
\begin{equation}\label{eq:5.6}
\begin{split}
f(v_{\ell,m,n})=\sqrt{q}^{3\ell+m-n}&\sum_{\substack{(i,j)\in\{1,3\}^2\\ i\neq j}}\left(F_{iij}z_i^{\ell+m}z_j^n
+F_{iji}z_i^{\ell+n}z_j^m+F_{ijj}z_i^\ell z_j^{m+n}\right),\\
\end{split}
\end{equation}
where
\begin{equation}\begin{split}\nonumber
F_{iij}=&\frac{(z_i-qz_j)^4\{1+q+(\ell-m)(1-q)\}\{1+q+n(1-q)\}}{(z_1-z_3)^4(q+1)(q^2+q+1)(q^3+q^2+q+1)}\\
F_{131}=&\frac{-(z_1-qz_3)^3(z_3-qz_1)\{1+q+(\ell-n)(1-q)\}\{1+q+m(1-q)\}}{(z_1-z_3)^4(q+1)(q^2+q+1)(q^3+q^2+q+1)}\\
&-\frac{(1-q)(1-q^2)(z_1-qz_3)^2\{2+2q+(\ell+m-n+1)(1-q)\}z_1z_3}{(z_1-z_3)^4(q+1)(q^2+q+1)(q^3+q^2+q+1)}\\
&-\frac{2(1-q)^2(1-q^2)(z_1-qz_3)z_1z_3^2}{(z_1-z_3)^4(q+1)(q^2+q+1)(q^3+q^2+q+1)}
\end{split}
\end{equation}
\begin{equation}
\begin{split}\nonumber
F_{313}=&\frac{-(z_3-qz_1)^3(z_1-qz_3)\{1+q+(\ell-n)(1-q)\}\{1+q+m(1-q)\}}{(z_1-z_3)^4(q+1)(q^2+q+1)(q^3+q^2+q+1)}\\
&-\frac{(1-q)(1-q^2)(z_3-qz_1)^2\{2+2q+(\ell+m-n-1)\}z_1z_3}{(z_1-z_3)^4(q+1)(q^2+q+1)(q^3+q^2+q+1)}\\
&-\frac{2(1-q)^2(1-q^2)(z_3-qz_1)z_1z_3^2}{(z_1-z_3)^4(q+1)(q^2+q+1)(q^3+q^2+q+1)}\\
F_{ijj}=&\frac{(z_1-qz_3)^2(z_3-qz_1)^2\{1+q+\ell(1-q)\}\{1+q+(m-n)(1-q)\}}{(z_1-z_3)^4(q+1)(q^2+q+1)(q^3+q^2+q+1)}\\
&+\frac{2(1-q)(1-q^2)(z_1-qz_3)(z_3-qz_1)\{1+q+(m-n)(1-q)\}z_1z_3}{(z_1-z_3)^4(q+1)(q^2+q+1)(q^3+q^2+q+1)}.\\
\end{split}
\end{equation}
\end{prop}
\begin{proof}Since for any $\ell\geq m\geq n,$
\begin{equation}
\begin{split}\nonumber
&\lim_{z_4\rightarrow z_3}D_{113}z_1^{\ell+m}z_3^n+D_{114}z_1^{\ell+m}z_4^n=F_{113}z_1^{\ell+m}z_3^n\\
&\lim_{z_4\rightarrow z_3}D_{131}z_1^{\ell+n}z_3^m+D_{141}z_1^{\ell+n}z_4^m=F_{131}z_1^{\ell+n}z_3^m\\
&\lim_{z_4\rightarrow z_3}D_{133}z_1^{\ell}z_3^{m+n}+D_{144}z_1^{\ell}z_4^{m+n}=F_{133}z_1^{\ell}z_3^{m+n}\\
&\lim_{z_4\rightarrow z_3}D_{311}z_1^{m+n}z_3^\ell+D_{411}z_1^{m+n}z_4^\ell=F_{311}z_1^{m+n}z_3^\ell\\
&\lim_{z_4\rightarrow z_3}D_{313}z_1^{m}z_3^{\ell+n}+D_{414}z_1^{m}z_4^{\ell+n}=F_{313}z_1^{m}z_3^{\ell+n}\\
&\lim_{z_4\rightarrow z_3}D_{341}z_1^{n}z_3^{\ell+m}+D_{441}z_1^{n}z_4^{\ell+n}=F_{431}z_1^{n}z_3^{\ell+m}\\
\end{split}
\end{equation}
by L'H\^opital's law, we have \eqref{eq:5.6}.
\end{proof}

\section{The proof of Lemma \ref{lem:6.77}}\label{ap:b}
In this section, we prove Lemma \ref{lem:6.77} for the case when $z_i\neq z_j$ since the proof is similar to the case when $z_i=z_j$ for some $i,j$. 

If a 3-tuple $(\lambda_1,\lambda_2,\lambda_3)$ is a simultaneous eigenvalue of $A_{w,i}$, there exists a function $h\in L^2(\Gamma\backslash\mathcal{B}_4)$ with $\|h\|_{2,w}=1$ such that for any $i\in\{1,2,3\}$, the function 
$$\delta_i(v_{\ell,m,n}):=A_{w,i}h(v_{\ell,m,n})-\lambda_ih(v_{\ell,m,n})$$ satisfies $\|\delta_i\|_{2,w}<\epsilon$. Let $f$ be a simultaneous eigenfunction corresponding to $(\lambda_1,\lambda_2,\lambda_3)$ with $f(v_{0,0,0})=h(v_{0,0,0}).$ Denote $$a_{\ell,m,n}=h(v_{\ell,m,n})-f(v_{\ell,m,n}).$$
The sequence $a_{\ell,m,n}$ is different from the sequence in Appendix \ref{sec:appendix}. 

Since $A_if(v_{\ell,m,n})=\lambda_if(v_{\ell,m,n})$ for any $i$, we have
\begin{equation}\label{eq:b.1}
\delta_i(v_{\ell,m,n})=A_{w,i}a_{\ell,m,n}-\lambda_ia_{\ell,m,n}.
\end{equation}
This implies that
\begin{equation*}
\begin{split}
a_{1,0,0}&=\frac{\delta_1(v_{0,0,0})}{q^3+q^2+q+1},\quad a_{1,1,0}=\frac{\delta_2(v_{0,0,0})}{q^4+q^3+2q^2+q+1},\quad a_{1,1,1}=\frac{\delta_3(v_{0,0,0})}{q^3+q^2+q+1}.
\end{split}
\end{equation*}
 Using \eqref{eq:b.1}, the first and the second equation of \eqref{eq:4.5}, we obtain
\begin{equation*}
\begin{split}
\delta_1(v_{1,0,0})&=(q^3+q^2+q)a_{1,1,0}+a_{2,0,0}-\lambda_1a_{1,0,0}\\
\delta_2(v_{1,0,0})&=(q^4+q^3+q^2)a_{1,1,1}+(q^2+q+1)a_{2,1,0}-\lambda_2a_{1,0,0}.
\end{split}
\end{equation*}
This shows that
\begin{equation}\label{bb}
\begin{split}
a_{2,0,0}&=\delta_1(v_{1,0,0})+\frac{\lambda_1\delta_1(v_{0,0,0})}{q^3+q^2+q+1}-\frac{q\delta_2(v_{0,0,0})}{q^2+1}\\
a_{2,1,0}&=\frac{\lambda_2\delta_1(v_{0,0,0})}{(q^3+q^2+q+1)(q^2+q+1)}+\frac{\delta_2(v_{1,0,0})}{q^2+q+1}-\frac{q^2\delta_3(v_{0,0,0})}{q^3+q^2+q+1}.
\end{split}
\end{equation}
Simliarly, we have
\begin{equation*}
\begin{split}
\delta_1(v_{2,0,0})=&(q^3+q^2+q)a_{2,1,0}+a_{3,0,0}-\lambda_1a_{2,0,0}\\
\end{split}
\end{equation*}
It follows from \eqref{bb} and the above equation that
\begin{equation*}
\begin{split}
a_{3,0,0}&=\delta_1(v_{2,0,0})+\lambda_1\delta_1(v_{1,0,0})+\frac{(\lambda_1^2-q\lambda_2)\delta_1(v_{0,0,0})}{q^3+q^2+q+1}\\
&-q\delta_2(v_{1,0,0})-\frac{q\lambda_1\delta_2(v_{0,0,0})}{q^2+1}+\frac{q^3(q^2+q+1)\delta_3(v_{0,0,0})}{q^3+q^2+q+1}.
\end{split}
\end{equation*}
By definition of $\delta_i$ and the equation \eqref{eq:4.5}, we have
\begin{equation*}
\begin{split}
&a_{\ell+2,0,0}+(q^3+q^2+q)a_{\ell+1,1,0}-\lambda_1a_{\ell+1,0,0}=\delta_1(v_{\ell+1,0,0})\\
&-(q^5+q^4+q^3)a_{\ell,1,1}-(q^3+q^2+q)a_{\ell+1,1,0}+q\lambda_2a_{\ell,0,0}=-q\delta_2(v_{\ell,0,0})\\
&q^6a_{\ell-2,0,0}+(q^5+q^4+q^3)a_{\ell,1,1}-q^3\lambda_3a_{\ell-1,0,0}=q^3\delta_3(v_{\ell-1,0,0}).
\end{split}
\end{equation*}
The sum of three equations is  
\begin{equation*}
\begin{split}
&a_{\ell+2,0,0}-\lambda_1a_{\ell+1,0,0}+q\lambda_2a_{\ell,0,0}-q^3\lambda_3a_{\ell-1,0,0}+q^6a_{\ell-2,0,0}\\
&=\delta_1(v_{\ell+1,0,0})-q\delta_2(v_{\ell,0,0})+q^3\delta_3(v_{\ell-1,0,0}).
\end{split}
\end{equation*}
Let $\{\alpha_n\}_n$ be a sequence defined by $\alpha_{-2}=\alpha_{-1}=\alpha_{0}=0$, $\alpha_1=1$ and
$$\alpha_n=\lambda_1\alpha_{n-1}-q\lambda_2\alpha_{n-2}+q^3\lambda_3\alpha_{n-3}-q^6\alpha_{n-4}$$
for any $n\geq 2.$
The characteristic equation of $\alpha_n$ shows that if $z_i\neq z_j$ for any $i,j$
$$\frac{\alpha_n}{q^{3n/2}}=X_1z_1^{n-1}+X_2z_2^{n-1}+X_3z_3^{n-1}+X_4z_4^{n-1}.$$
Here $X_i=\frac{z_i^3}{(z_i-z_j)(z_i-z_k)(z_i-z_l)}$ where $i,j,k,l$ are distinct numbers in $\{1,2,3,4\}.$ 

By the characteristic equation of $a_{\ell,0,0}$ and induction on $\ell$, for any $\ell\geq 4,$ the sequence $a_{\ell,0,0}$ satisfies that
\begin{equation}\label{eq:b.3}
\begin{split}
\frac{a_{\ell,0,0}}{q^{3\ell/2}}=\sum_{i=1}^\ell\frac{c_{\ell-i}}{q^{3(\ell-i)/2}}\frac{\alpha_i}{q^{3i/2}}=\sum_{j=1}^4\sum_{i=1}^\ell\frac{c_{\ell-i}}{q^{3(\ell-i)/2}}X_jz_j^{i-1}
\end{split}
\end{equation}
where $c_n$ is a sequence defined by 
\begin{equation*}
c_n=\begin{cases}\frac{\delta_1(v_{0,0,0})}{q^3+q^2+q+1}&\text{ if } n=0\\
\delta_1(v_{1,0,0})-\frac{q\delta_2(v_{0,0,0})}{q^2+1}&\text{ if } n=1\\
\delta_1(v_{2,0,0})-{{q}}\delta_2(v_{1,0,0})+\frac{q^3(q^2+q+1)\delta_3(v_{0,0,0})}{q^3+q^2+q+1}&\text{ if }n=2\\
\delta_1(v_{n,0,0})-{{q}}\delta_2(v_{n-1,0,0})+q^3\delta_3(v_{n-2,0,0})&\text{ otherwise}.
\end{cases}
\end{equation*}
By the choice of $\delta_i$, there exists a constant $\mathcal{C}_1$ such that for any $i$, $\left|\frac{c_i}{q^{3i/2}}\right|< \mathcal{C}_1\epsilon$. Using this, we have for any $c>0$,
\begin{equation*}
\begin{split}
\left|\frac{a_{\ell,0,0}}{q^{3(\ell+c)/2}}\right|
\leq \mathcal{C}_1\epsilon\sum_{j=1}^4\sum_{i=0}^\ell|X_j|\left|\frac{z_j}{q^{c/2}}\right|^{i-1}\leq \frac{2\mathcal{C}_1q^{c/2}\epsilon|X_1|}{|z_1|-q^{c/2}}\left|\frac{z_1}{q^{c/2}}\right|^{\ell}+o(|z_1|^\ell).
\end{split}
\end{equation*}
Since $\left|\frac{f(v_{\ell,0,0})}{q^{(3+c)\ell/2}}\right|\geq |C_1|\left|\frac{z_1}{q^{c/2}}\right|^\ell-o(|z_1|^{\ell})$, we have
$$\left|\frac{h(v_{\ell,0,0})}{q^{(3+c)\ell/2}}\right|\geq \left|\frac{f(v_{\ell,0,0})}{q^{(3+c)\ell/2}}\right|-\left|\frac{a_{\ell,0,0}}{q^{(3+c)\ell/2}}\right|\geq \left(|C_1|-\frac{2\mathcal{C}_1q^{c/2}\epsilon|X_1|}{|z_1|-q^{c/2}}\right)\left|\frac{z_1}{q^{c/2}}\right|^\ell-o(|z_1|^\ell).$$
Let $\{\beta_n\}_n$ be a sequence defined by $d_{-2}=d_{-1}=d_{0}=0$, $d_1=1$ and
$$\beta_n=\lambda_3\beta_{n-1}-q\lambda_2\beta_{n-2}+q^3\lambda_1\beta_{n-3}-q^6\beta_{n-4}$$
The characteristic equation of $\beta_n$ shows that if $z_i\neq z_j$ for any $i,j$
$$\frac{\beta_n}{q^{3n/2}}=X_{123}(z_1z_2z_3)^{n-1}+X_{124}(z_1z_2z_4)^{n-1}+X_{134}(z_1z_3z_4)^{n-1}+X_{234}(z_2z_3z_4)^{n-1}.$$
Here $X_{ijk}=\frac{(z_l)^{-3}}{(z_l^{-1}-z_i^{-1})(z_l^{-1}-z_j^{-1})(z_l^{-1}-z_k^{-1})}$ where $i,j,k,l$ are distinct numbers in $\{1,2,3,4\}$.
Similarly, using the equations in \eqref{eq:4.7}, we have that for any $\ell\geq 4$,
\begin{equation}\label{eq:b..4}
\begin{split}
\frac{a_{\ell,\ell,\ell}}{q^{3\ell/2}}=\sum_{i=1}^\ell\frac{d_{\ell-i}}{q^{3(\ell-i)/2}}\frac{\beta_i}{q^{3i/2}}
\end{split}
\end{equation}
where $d_i$ is a sequence defined by 
\begin{equation*}d_n=\begin{cases}\frac{\delta_3(v_{0,0,0})}{q^3+q^2+q+1}&\text{ if }n=0\\
\delta_3(v_{1,1,1})-\frac{q\delta_2(v_{0,0,0})}{q^2+1}&\text{ if }n=1\\
\delta_3(v_{2,2,2})-{{q}}\delta_2(v_{1,1,1})+\frac{q^3(q^2+q+1)\delta_1(v_{0,0,0})}{q^3+q^2+q+1}&\text{ if } n=2\\
\delta_3(v_{n,n,n})-{{q}}\delta_2(v_{n-1,n-1,n-1})-q^3\delta_3(v_{n-2,n-2,n-2})&\text{ otherwise}.
\end{cases}
\end{equation*}
Analogous to the case $a_{\ell,0,0}$, there exists a constant $\mathcal{C}_3$ such that
$$\left|\frac{h(v_{\ell,\ell,\ell})}{q^{(3+c)\ell/2}}\right|
\geq \left(|C_3|-\frac{2\mathcal{C}_3q^{c/2}\epsilon|X_{123}|}{|z_1z_2z_3|-q^{c/2}}\right)\left|\frac{z_1z_2z_3}{q^{c/2}}\right|^\ell-o(|z_1z_2z_3|^\ell).$$

The remaining part of this section is the case $a_{\ell,\ell,0}$. By the first equation and the third equation in \eqref{eq:4,8}, we have
\begin{equation*}
\begin{split}
&-q^4(q+1)a_{\ell-1,\ell-2,1}-q^3a_{\ell-1,\ell-1,0}-q^2a_{\ell,\ell-2,0}+q^2\lambda_1a_{\ell-1,\ell-2,0}=-q^2\delta_1(v_{\ell-1,\ell-2,0})\\
&q^3a_{\ell-1,\ell-1,0}+q^2a_{\ell,\ell-2,0}+(q+1)a_{\ell+1,\ell,1}-\lambda_3a_{\ell,\ell-1,0}=\delta_3(v_{\ell,\ell-1,0}).
\end{split}
\end{equation*}
The sum of two equations is  
\begin{equation}\label{eq:b.5}
\begin{split}
&(q+1)a_{\ell+1,\ell,1}-\lambda_3a_{\ell,\ell-1,0}+q^2\lambda_1a_{\ell-1,\ell-2,0}-q^4(q+1)a_{\ell-1,\ell-2,1}\\
&=\delta_3(v_{\ell,\ell-1,0})-q^2\delta_1(v_{\ell-1,\ell-2,0}).
\end{split}
\end{equation}
It follows from the first and third equation in \eqref{eq:4.6} that
\begin{equation}\label{eq:b..6}
\begin{split}
&(q+1)a_{\ell+1,\ell,0}-\lambda_1a_{\ell,\ell,0}-q^4(q+1)a_{\ell-1,\ell-2,0}+\lambda_3q^2a_{\ell-1,\ell-1,0}\\
&=\delta_1(v_{\ell,\ell,0})-q^2\delta_3(v_{\ell-1,\ell-1,0}).
\end{split}
\end{equation}
Using \eqref{eq:b.5} and \eqref{eq:b..6}, we have the following three equations:
\begin{itemize}
\item multiplying $-q(q+1)$ by \eqref{eq:b.5},
\item changing $\ell$ to $\ell-1$ of \eqref{eq:b..6} and multiplying $-\lambda_3q$ and
\item changing $\ell$ to $\ell-2$ and multiplying $\lambda_1q^3$. 
\end{itemize}
The sum of three equations is  
\begin{equation}\label{eq:b.6}
\begin{split}
&-q(q+1)^2a_{\ell+1,\ell,1}+\lambda_1\lambda_3qa_{\ell-1,\ell-1,0}+q^5(q+1)^2a_{\ell-1,\ell-2,1}+\lambda_3q^5(q+1)a_{\ell-2,\ell-3,0}\\
&-(\lambda_1^2+\lambda_3^2)q^3a_{\ell-2,\ell-2,0}+\lambda_1\lambda_3q^5a_{\ell-3,\ell-3,0}-\lambda_1q^7(q+1)a_{\ell-3,\ell-4,0}\\
&=-q(q+1)\delta_3(v_{\ell,\ell-1,0})+q^3(q+1)\delta_1(v_{\ell-1,\ell-2,0})-\lambda_3q\delta_1(v_{\ell-1,\ell-1,0})\\
&+\lambda_3q^3\delta_3(v_{\ell-2,\ell-2,0})+\lambda_1q^3\delta_1(v_{\ell-2,\ell-2,0})-\lambda_1q^5\delta_3(v_{\ell-3,\ell-3,0})
\end{split}
\end{equation}
The sum of \eqref{eq:b.6} and the equation obtained by multiplying \eqref{eq:b.5} by $q^5(q+1)$ and changing $\ell$ to $\ell-2$ is
\begin{equation}\label{eq:b.7}
\begin{split}
&-q(q+1)^2a_{\ell+1,\ell,1}+\lambda_1\lambda_3qa_{\ell-1,\ell-1,0}+2q^5(q+1)^2a_{\ell-1,\ell-2,1}\\
&-(\lambda_1^2+\lambda_3^2)q^3a_{\ell-2,\ell-2,0}+\lambda_1\lambda_3q^5a_{\ell-3,\ell-3,0}-q^9(q+1)a_{\ell-3,\ell-4,1}\\
&=-q(q+1)\delta_3(v_{\ell,\ell-1,0})+q^3(q+1)\delta_1(v_{\ell-1,\ell-2,0})\\
&+q^5(q+1)\delta_3(v_{\ell-2,\ell-3,0})-q^7(q+1)\delta_1(v_{\ell-3,\ell-4,0})\\
&-\lambda_3q\delta_1(v_{\ell-1,\ell-1,0})+\lambda_3q^3\delta_3(v_{\ell-2,\ell-2,0})+\lambda_1q^3\delta_1(v_{\ell-2,\ell-2,0})-\lambda_1q^5\delta_3(v_{\ell-3,\ell-3,0})
\end{split}
\end{equation}
The second equation of \eqref{eq:4.6} shows that
\begin{equation}\label{eq:b.8}
a_{\ell+1,\ell+1,0}+q(a+1)^2a_{\ell+1,\ell,1}-\lambda_2a_{\ell,\ell,0}+q^4a_{\ell-1,\ell-1,0}=\delta_2(v_{\ell,\ell,0}).
\end{equation}
Using \eqref{eq:b.8}, it is possible to eliminate the terms of the form $a_{\ell+1,\ell,1}$ in \eqref{eq:b.7}. Then for any $\ell\geq 5$, we have
\begin{equation}\label{eq:b.10}
\begin{split}
&a_{\ell+1,\ell+1,0}-\lambda_2a_{\ell,\ell,0}+(\lambda_1\lambda_3-q^3)qa_{\ell-1,\ell-1,0}-(\lambda_1^2+\lambda_3^2-2\lambda_2q)q^3a_{\ell-2,\ell-2,0}\\
&+(\lambda_1\lambda_3-q^3)q^5a_{\ell-3,\ell-3,0}-\lambda_2q^8a_{\ell-4,\ell-4,0}+q^{12}a_{\ell-5,\ell-5,0}\\
&=\delta_2(v_{\ell,\ell,0})-2q^4\delta_2(v_{\ell-2,\ell-2,0})+q^8\delta_2(v_{\ell-4,\ell-4,0})-q(q+1)\delta_3(v_{\ell,\ell-1,0})\\
&+q^3(q+1)\delta_1(v_{\ell-1,\ell-2,0})-\lambda_3q\delta_1(v_{\ell-1,\ell-1,0})+\lambda_3q^3\delta_3(v_{\ell-2,\ell-2,0})+\lambda_1q^3\delta_1(v_{\ell-2,\ell-2,0})\\
&-\lambda_1q^5\delta_3(v_{\ell-3,\ell-3,0})+q^5(q+1)\delta_3(v_{\ell-2,\ell-3,0})-q^7(q+1)\delta_1(v_{\ell-3,\ell-4,0}).
\end{split}
\end{equation}
 for any $\ell\geq5$, let $e_{\ell+1}$ be the right hand side of \eqref{eq:b.10}.

Let $\{\gamma_n\}_n$ be a sequence defined by $\gamma_{-n}=0$, $\gamma_1=1$ and
\begin{equation*}
\begin{split}
\gamma_n=&\lambda_2\gamma_{n-1}-(\lambda_1\lambda_3-q^3)q\gamma_{n-2}+(\lambda_1^2+\lambda_3^2-2\lambda_2q)q^3\gamma_{n-3}\\
&-(\lambda_1\lambda_3-q^3)q^5\gamma_{n-4}+\lambda_2q^8\gamma_{n-5}+q^{12}\gamma_{n-6}
\end{split}
\end{equation*}
for any $n>0$. 
Then $\gamma_n$ satisfies that 
\begin{equation*}
\begin{split}
\gamma_n=X_{12}(z_1z_2)^{n-1}&+X_{13}(z_1z_3)^{n-1}+X_{14}(z_1z_4)^{n-1}\\
&+X_{23}(z_2z_3)^{n-1}+X_{24}(z_2z_4)^{n-1}+X_{34}(z_3z_4)^{n-1}.
\end{split}
\end{equation*}
Here $X_{ij}=\frac{z_i^5z_j^5}{(z_iz_j-z_iz_k)(z_iz_j-z_iz_k)(z_iz_j-z_iz_l)(z_iz_j-z_jz_k)(z_iz_j-z_jz_l)}$ where $i,j,k,l$ are distinct numbers in $\{1,2,3,4\}.$

Using \eqref{eq:b.5}, \eqref{eq:b.6} and the second equation of \eqref{eq:4.6}, we have a constants $e_0,e_1,\dots ,e_5$, and $\mathcal{C}_3$ such that 
\begin{equation}\label{eq:b.13}
\frac{a_{\ell,\ell,0}}{q^{2\ell}}=\sum_{i=1}^\ell\frac{e_{\ell-i}}{q^{2(\ell-i)}}\frac{\gamma_i}{q^{2i}}.
\end{equation}
for any $\ell<5$ and $\left|\frac{e_\ell}{q^{2\ell}}\right|<\epsilon \mathcal{C}_2$ for any $\ell$.
Using induction on $\ell$ and \ref{eq:b.10}, the equation \eqref{eq:b.13} holds for any $\ell$.

Analogous to the case $a_{\ell,0,0}$ and $a_{\ell,\ell,\ell}$, we obtain
$$\left|\frac{h(v_{\ell,\ell,0})}{q^{(2+c)\ell}}\right|
\geq \left(|C_2|-\frac{2\mathcal{C}_2q^{c}\epsilon|X_{12}|}{|z_1z_2|-q^{c}}\right)\left|\frac{z_1z_2}{q^{c}}\right|^\ell-o(|z_1z_2|^\ell).$$

\end{document}